\documentclass{article}

\usepackage[margin=1in]{geometry}
\usepackage[utf8]{inputenc}
\usepackage{amsmath,amsfonts,amsthm,amssymb}
\usepackage{thmtools}
\usepackage{thm-restate}
\usepackage{graphicx}
\usepackage{caption}
\usepackage{subcaption}
\usepackage{xcolor}
\usepackage{pgfplots}
\usepackage{hyperref}
\hypersetup{colorlinks=true,citecolor=blue,linkcolor=blue,filecolor=blue,urlcolor=blue,breaklinks=true}
\usepackage{cleveref}
\usepackage{todonotes}

\declaretheorem[name=Theorem]{theorem}

\newtheorem*{theorem*}{Theorem}
\newtheorem{lemma}{Lemma}
\newtheorem{cor}{Corollary}

\newtheorem{conjecture}{Conjecture}

\theoremstyle{definition}
\newtheorem{definition}{Definition}

\newtheorem{remark}{Remark}

\newcommand{\cR}{\mathcal{R}}

\newcommand{\R}{\mathbb{R}}
\renewcommand{\H}{\mathbb{H}}

\newcommand{\N}{\mathbb{N}}

\renewcommand{\d}{\mathrm{d}}
\renewcommand{\Re}{\operatorname{Re}}
\renewcommand{\Im}{\operatorname{Im}}

\newcommand{\RR}{\R}

\DeclareMathOperator{\tr}{Tr}

\DeclareMathOperator{\supp}{supp}
\DeclareMathOperator*{\argmin}{argmin}

\newcommand{\codeurl}{\url{https://www.github.com/oisinfaust/alpha-divergence-quad}}

\newcommand{\psd}{\succeq}
\newcommand{\nsd}{\preceq}

\title{Rational approximations of operator monotone\\
and operator convex functions}
\author{Ois\'in Faust\thanks{Department of Applied Mathematics and Theoretical Physics, University of Cambridge, United Kingdom} \and Hamza Fawzi\protect\footnotemark[1]}
\date{\today}

\begin{document}

\maketitle

\begin{abstract}
Operator convex functions defined on the positive half-line play a prominent role in the theory of quantum information, where they are used to define quantum $f$-divergences. Such functions admit integral representations in terms of rational functions. Obtaining high-quality rational approximants of operator convex functions is particularly useful for solving optimization problems involving quantum $f$-divergences using semidefinite programming. In this paper we study the  quality of rational approximations of operator convex (and operator monotone) functions.  Our main theoretical results are precise global bounds on the error of local Pad\'e-like approximants, as well as minimax approximants, with respect to different weight functions. While the error of Pad\'e-like approximants depends inverse polynomially on the degree of the approximant, the error of minimax approximants has root exponential dependence and we give detailed estimates of the exponents in both cases. We also explain how minimax approximants can be obtained in practice using the differential correction algorithm. 
\end{abstract}

\section{Introduction}

Matrix functions have countless applications in applied mathematics \cite{higham2008functions}.
Given a function $f:I\to \RR$ defined on an interval $I$ of $\RR$, one can extend $f$ to act on Hermitian matrices by applying $f$ to the eigenvalues. More precisely, if $A$ is a Hermitian matrix (of any finite size) with spectral decomposition
\[
A = \sum_{i} \lambda_i v_i v_i^{\dagger}
\]
where $\lambda_i \in I$, and $\{v_i\}$ is an orthonormal family of eigenvectors, we define $f(A)$ by
\[
f(A) = \sum_{i} f(\lambda_i) v_i v_i^{\dagger}.
\]

\paragraph{Operator monotone and operator convex functions} The space of Hermitian matrices is equipped with a partial order, known as the L{\"o}wner order whereby $A \psd B$ if and only if $A-B$ is positive semidefinite. In his seminal 1934 paper, L{\"o}wner \cite{lowner1934monotone} introduced and characterized so-called \emph{operator monotone} functions $h:I\to \RR$ which satisfy
\[
A\psd B \implies h(A) \psd h(B)
\]
for all Hermitian matrices $A,B$ of any size, whose spectra lie in $I$. He showed that the class of operator monotone functions coincides precisely with the class of \emph{Pick functions} from complex analysis which admit an analytic continuation to the open upper half plane. Importantly, such functions admit an integral representation in terms of rational functions. In the case where $h$ is defined on $I=(0,\infty)$, which will be the main setting of this paper, L{\"o}wner's theorem asserts that one can write
\begin{equation}
\label{eq:intrepopmon}
h(x) = h(1) + \int_{0}^{1} \frac{x-1}{1+t(x-1)} d\nu(t)
\end{equation}
for some finite measure $\nu$ supported on $[0,1]$. For each $t \in [0,1]$ the rational integrand (in $x$) is operator monotone, and L{\"o}wner's theorem asserts that any operator monotone function is essentially a positive linear combination of such rational functions. Prominent examples of operator monotone functions are $h(x) = \log x$, and $h(x) = x^{\alpha}$ for $\alpha \in [0,1]$, the latter example being known as the L{\"o}wner-Heinz inequality.
Closely related to operator monotone functions are \emph{operator convex} functions $f:I\to \RR$ which satisfy Jensen's inequality in the L{\"owner} order
\[
f(\lambda A + (1-\lambda)B) \nsd \lambda f(A) + (1-\lambda) f(B),
\]
for all $\lambda \in [0,1]$ and all Hermitian matrices $A,B$ having a spectrum contained in $I$. Such functions were studied by L{\"o}wner's doctoral student Kraus in 1936 \cite{kraus1936konvexe}, where he established a  characterization similar to the above. Any operator convex function $f:(0,\infty)\to \RR$ can be expressed as
\begin{equation}
\label{eq:intrepopcvx}
f(x) = f(1) + f'(1)(x-1) + \int_{0}^{1} \frac{(x-1)^2}{1+t(x-1)} d\mu(t)
\end{equation}
where $\mu$ is a finite measure supported on $[0,1]$. Examples of operator convex functions are $f(x) = x\log x$ and $f(x) = x^{\alpha}$ for all $\alpha \in [1,2]$. We note that all operator monotone functions \eqref{eq:intrepopmon} are necessarily operator \emph{concave}, however the converse is not true.

\paragraph{Quantum $f$-divergences} Operator convexity plays a crucial role in the area of quantum information theory. A \emph{density matrix} is a Hermitian positive semidefinite matrix with trace equal to 1. Density matrices are the quantum analogue of classical probability distributions, and represent probabilistic mixtures of quantum states. If $\rho$ and $\sigma$ are two density matrices, a fundamental quantity in quantum information is the \emph{quantum relative entropy} defined by
\begin{equation}
\label{eq:qre}
S(\rho \| \sigma) = \tr[\rho(\log \rho - \log \sigma)],
\end{equation}
which is the quantum counterpart of the classical Kullback-Leibler divergence.
More generally for $\alpha \in (1,2]$, the $\alpha$-quasi-entropy of the pair $(\rho,\sigma)$ is defined by
\begin{equation}
\label{eq:Salpha}
S_{\alpha}(\rho \| \sigma) = \frac{1}{\alpha-1}(\tr[\rho^{\alpha} \sigma^{1-\alpha}]-\tr \sigma)
\end{equation}
which converges to $S(\rho\|\sigma)$ as $\alpha\to 1$. A key fact about $S_{\alpha}$ and $S$ is that they are joint convex functions in $(\rho,\sigma)$; this is a (nontrivial) consequence of the operator convexity of the functions $x^{\alpha}$ for $\alpha \in [1,2]$, see \cite{lieb1973convex,Lindblad1974}. The $\alpha$-quasi entropy defined above is only a special case of so-called quantum $f$-divergences \cite{petz1986quasi}, defined for any operator convex $f:(0,\infty)\to \RR$, whose precise definition we omit here. Let us just mention that these are the quantum analogues of the well-known $f$-divergences defined in classical probability and information theory for probability distributions $p=\{p_i\}$ and $q=\{q_i\}$ via the expression
\[
S_f(p\|q) = \sum_{i} q_i f(p_i/q_i),
\]
which is convex in $(p,q)$ for any choice of convex function $f:(0,\infty)\to \RR$.

\paragraph{Optimization and semidefinite programming} Many problems in quantum information are naturally expressed as optimization problems involving a quantum $f$-divergence, and in particular the quantum entropies \eqref{eq:qre} or \eqref{eq:Salpha}. This includes for example the problem of evaluating the efficiency of a quantum key distribution protocol in cryptography \cite{winick2018reliable}, measuring the amount of entanglement in a quantum state \cite{zinchenko2010numerical}, or the evaluation of quantum channel capacities \cite{shor2003capacities}. Given the complex nature of some of these optimization problems, it is highly desirable to express them in a standard form for which efficient and reliable algorithms exist. \emph{Semidefinite programming} \cite{vandenberghe1996semidefinite} has emerged as a natural way to formulate convex optimization problems arising in quantum information theory, given its ability to deal with Hermitian positive semidefinite variables. A semidefinite program is a convex optimization problem of the form
\begin{equation}
\label{eq:sdp}
\min_{x \in \RR^n} \quad c^T x \quad : \quad A_0 + x_1 A_1 + \dots + x_n A_n \psd 0
\end{equation}
where $c \in \RR^n$, and $A_0,A_1,\ldots,A_n$ are given Hermitian matrices. The constraint \eqref{eq:sdp} in a semidefinite program is known as a \emph{linear matrix inequality} and it describes a convex region in $\RR^n$. Semidefinite programs can be solved efficiently using a variety of algorithms such as  interior-point methods \cite{vandenberghe1996semidefinite} or first-order splitting methods \cite{scspaper}. Optimization problems involving the quantum entropy function \eqref{eq:qre} however cannot be directly expressed in semidefinite form since the feasible set of a semidefinite optimization problem is necessarily semialgebraic \cite{blekherman2012semidefinite}, while the quantum entropy function is not. One approach around this problem is to work with rational approximations of the entropy function. This approach was adopted in \cite{fawzi2019semidefinite,brown2021device} where the approximations were obtained from quadrature rules applied to the integral representations \eqref{eq:intrepopmon} and \eqref{eq:intrepopcvx}. The key fact is that while a general operator convex function $f$ may not be amenable to semidefinite programming, the rational integrand 
\[
f_t: x\mapsto \frac{(x-1)^2}{1+t(x-1)}
\]
is. Indeed, observe that a convex constraint of the form
\[
f_t(x) \leq \tau
\]
can be equivalently described by the $2\times 2$ linear matrix inequality
\[
\begin{bmatrix}
1+t(x-1) & x-1\\
x-1 & \tau
\end{bmatrix}
\psd 0.
\]
Such a \emph{semidefinite programming representation} of $f_t$ can be extended to any finite positive sum of $\{f_{t_i}\}$. Furthermore, with some additional (nontrivial) work these representations can be extended to matrix arguments, and to quantum $f$-divergences as shown in \cite{fawzi2017lieb,fawzi2019semidefinite,fawzi2023smaller}.

As such, it is of significant interest to understand how to best approximate an operator monotone or convex function by discretizing the integral \eqref{eq:intrepopmon} or \eqref{eq:intrepopcvx}, i.e., (in the case of an operator convex functions)
\[
f(x) \approx \sum_{i=1}^m u_i \frac{(x-1)^2}{1+t_i(x-1)}
\]
for some weights $u_i > 0$ and nodes $t_i \in [0,1]$. Such approximations can in turn be used to approximate quantum $f$-divergences (such as the quantum relative entropy \eqref{eq:qre}) via functions that admit a semidefinite programming representation.

In \cite{fawzi2019semidefinite} it was observed that applying Gaussian quadrature to the integral \eqref{eq:intrepopmon}, with respect to the measure $d\nu(t)$, yields a diagonal Pad\'e approximant to the function $h$. One drawback of this approximation is that it is neither an upper bound, nor a lower bound on $h$, a feature which is often desired in optimization. Later, it was realized in \cite{brown2021device,fawzi2022semidefinite} that if one uses the Gauss-Radau quadrature instead for $h(x)=\log x$ then one obtains rigorous  upper/lower bounds.

\paragraph{Main contributions} The goal of this paper is to systematically study rational approximations of operator monotone and operator convex functions defined on the positive half-line, with a view towards applications in semidefinite optimization and quantum information. We study two types of approximations and we precisely quantify the approximation errors of each type.
\begin{itemize}
\item We first study Gaussian quadrature-based approximations, where the integrals \eqref{eq:intrepopmon} and \eqref{eq:intrepopcvx} are discretized via some Gaussian quadrature rule. We show that by choosing suitable quadrature rules, one obtains upper/lower bounds on the function that are \emph{locally optimal around $x=1$}, i.e., they agree with the Taylor expansion to the higher possible order and coincide with certain Pad\'e approximants. We further quantify the approximation error as a function of the number of discretization points. Our first main theorem is stated below for operator monotone functions---a version for operator convex functions appears later in Theorem \ref{thm:op-convex-pade}. 
Recall that an $m$-point Gauss-Radau quadrature rule requires one of the nodes to be an endpoint of the integration interval, and is exact for all polynomials of degree up to $2m-2$ (see Section \ref{sec:prelim:gauss} for the precise definition).

\begin{theorem}\label{thm:op-mono-pade}
    Let $\nu$ be a finite Borel measure on $[0,1]$, and let $h_\nu(x)=\int_{0}^{1} \frac{(x-1)}{1+t(x-1)}\d\nu(t)$ be a corresponding operator monotone function satisfying $h_{\nu}(1) = 0$. 
    Let $\nu^0_m$ and $\nu_m^{1}$ be discrete measures associated to the $m$-node Gauss-Radau quadrature rules for $\nu$ with fixed node at 0 and at 1, respectively. 
    Let $h_{\nu^0_m}$ and $h_{\nu_m^{1}}$ be the rational operator monotone approximations to $h_\nu$ arising from $\nu^0_m$ and $\nu_m^{1}$. Then $h_{\nu_m^0}$ is a $[m/m-1]$ rational function, $h_{\nu_m^1}$ is a $[m/m]$ rational function, and we have:
    \begin{enumerate}
	    \item For $m\geq 1$,
	    \begin{equation}
	    \label{eq:hnu01bnds}
h_{\nu_m^0}(x) \ge h_{\nu_{m+1}^0}(x)\ge h_\nu(x) \geq h_{\nu_{m+1}^1}(x) \geq h_{\nu_{m+1}^1}(x).
		\end{equation}
		\item Locally around $x=1$,
		\begin{equation}
		\label{eq:hpade}
		\begin{aligned}
		h_{\nu_m^0}(x) - h_{\nu}(x) &= O((x-1)^{2m})\\
		h_{\nu_m^1}(x) - h_{\nu}(x) &= O((x-1)^{2m}).
		\end{aligned}
		\end{equation}
        \item For any $x > 0$,
        \begin{equation}
        \label{eq:gaussianglobalerror}
        h_{\nu_{m}^{0}}(x)-h_{\nu_m^1}(x) \le \max\{\nu_m^{0}(\{0\}),\nu_m^{1}(\{1\})\}\frac{(x-1)^2}{x}.
        \end{equation}
    \end{enumerate}
\end{theorem}
Equation \eqref{eq:hnu01bnds} says that the sequence of functions $(h_{\nu_m^0})$ (resp. $(h_{\nu_m^1})$) is monotonic nonincreasing (resp. nondecreasing), and is an upper bound (resp. lower bound) on $h_{\nu}$. Equation \eqref{eq:hpade} asserts that $h_{\nu^0_m}$ is the order $[m/m-1]$ \emph{Pad\'e approximant} to $h_\nu(x)$ at $x=1$, and $xh_{\nu_m^{1}}(x)$ is the order $[m/m-1]$ Pad\'e approximant to $xh_\nu(x)$ at $x=1$.\footnote{We should mention that monotonicity results along the lines of \eqref{eq:hnu01bnds} are well known \cite{gilewicz2010100} for the Pad\'e approximants of \emph{Stieltjes functions}, a class of functions which bear a strong resemblance to operator monotone functions.} Most importantly for us, \eqref{eq:gaussianglobalerror} gives a global approximation bound on the gap 
\[
h_{\nu_{m}^{0}}(x)-h_{\nu_m^1}(x) = (h_{\nu_{m}^{0}}(x) - h_{\nu}(x)) + (h_{\nu}(x) - h_{\nu_{m}^{1}}(x)),
\]
in terms of the weight of the endpoint in the Gauss-Radau quadrature rule, and relative to the function $(x-1)^2/x$. Since $h_\nu(x)$ can be interpreted as an average of the functions $\{x\mapsto \frac{x-1}{1+t(x-1)}\,:\,t\in[0,1]\}$ which are pointwise decreasing in $t$, a natural choice of function relative to which error can be measured is the difference between the maximum and minimum of these functions: $(x-1) - (x-1)/x=(x-1)^2/x$. In Section \ref{sec:quad} we work out the explicit values of $\nu_m^0(\{0\})$ and $\nu_m^1(\{1\})$ for the important example of $\alpha$-divergences which allows us to show that the convergence rate is given by $\approx 1/m^{2(1-|\alpha|)}$ for $\alpha \in (-1,1)$.


\item Theorem \ref{thm:op-mono-pade} 
quantifies the global accuracy of the best local approximants around $x=1$.
A natural question is to understand which approximants satisfy a global bound of the form \eqref{eq:gaussianglobalerror} with the \emph{best possible} dependence on $m$.
 In other words, given an operator convex function $f:(0,\infty)\to \RR$, and a nonnegative weight function $b:(0,\infty)\to \RR_{\geq 0}$, we seek to characterize the quantity
\begin{equation}
\label{eq:Em1m2}
E_{m_1,m_2} = \inf_{r \in \cR_{m_1,m_2}} \sup_{x \in (0,\infty)} \frac{|f(x) - r(x)|}{b(x)}
\end{equation}
where $\cR_{m_1,m_2}$ is the set of rational functions that can be expressed as $p(x)/q(x)$ where $\deg p \leq m_1$ and $\deg q \leq m_2$.
Leveraging existing results on best rational approximations we first prove, under mild conditions on the weight function $b$, that the best rational approximant in \eqref{eq:Em1m2} exists and can be obtained by applying a suitable discretization of the integral representation \eqref{eq:intrepopcvx}.
\begin{theorem}[See Theorem \ref{thm:op-convex-quad} for details]
Let $f:(0,\infty)\to \RR$ be operator convex with $f(1)=f'(1)=0$ and let $b:(0,\infty)\to\RR$ be a continuous weight function which is positive except at $x=1$. Under conditions \eqref{eq:h-cont-mid}-\eqref{eq:h-slowly-vanish} the best order $[m+1/m]$ rational approximation to $f$ relative to $b$ exists and has the form
\[
\tilde{f}(x) = \sum_{i=1}^m u_i\frac{(x-1)^2}{1+t_i(x-1)}
\]
for weights $u_i \geq 0$ and $0\leq t_1< \dots < t_m \leq 1$.
\end{theorem}
Next, we focus on the nonnegative operator convex functions
\[
f_{\alpha}(x) = \frac{1}{\alpha(\alpha-1)} (x^{\alpha}-\alpha(x-1)-1)
\]
which generate the so-called $\alpha$-divergences for $\alpha \in [-1,2]$. For $\alpha=0$ and $\alpha=1$ we have
\[
f_0(x) = -\log x - x+1, \quad f_1(x) = x\log x -x +1.
\]
We define for $0\leq \alpha \leq \beta\leq 2$ the quantity
\begin{equation}\label{eq:def-error-alpha}
    \epsilon^{[m]}_{\alpha,\beta}:= \inf_{\substack{0\le t_1<\dots<t_m\le 1\\ u_i\ge0,\,i=1,\dots,m}} \left\{\,\sup_{x\in(0,\infty)}\,
\Bigl|\frac{f_\alpha(x) - \tilde{f}(x)}{f_\beta(x)}\Bigr| \; :\; \tilde{f}(x)=\sum_{i=1}^m \frac{u_i(x-1)^2}{1+t_i(x-1)}\right\}.
\end{equation}
Our results quantify the behaviour of $\epsilon^{[m]}_{\alpha,\beta}$ as $m\to \infty$.

Note that by choosing the $\{u_i,t_i\}$ in \eqref{eq:def-error-alpha} via Gaussian quadrature (as in Theorems \ref{thm:op-mono-pade} and \ref{thm:op-convex-pade}), we can get \emph{upper bounds} on $\epsilon^{[m]}_{\alpha,\beta}$. However these upper bounds turn out to be far from tight. For example, one can show that Gaussian-quadrature based approximations yield upper bounds of the form $\epsilon^{[m]}_{\alpha,\beta} \lesssim C m^{-k}$  for some constants $C$ and $k$ that depend on $\alpha,\beta$ (for certain values of $\alpha,\beta$). 
As the next result shows, this inverse polynomial dependence on $m$ is far from optimal. Our first theorem concerns the case $\alpha \in (0,1)$, and shows that we can get root exponential instead.
\begin{restatable}{theorem}{thmapproxalphaopt}\label{thm:approx-alpha-opt}
For each $\alpha\in(0,1)$, there is a constant $C_\alpha>0$ such that
\[
\epsilon_{\alpha,\alpha}^{[m]} \le C_\alpha e^{-2\pi\sqrt{\alpha(1-\alpha)m}}
\]
for all $m\in\N$.
\end{restatable}
The decay rate in $\exp(-c\sqrt{m})$ is well-known to approximation theorists and is due to the presence of singularities. (Analytic functions can be approximated at a rate $\exp(-cm)$ by polynomials.) The constant $\sqrt{\alpha(1-\alpha)}$ comes from the presence of \emph{two} singularities for $f_{\alpha}$, at $x=0$ and $x=\infty$.

Our second theorem concerns the case $\alpha \notin [0,1]$.
\begin{restatable}{theorem}{thmalphabetasubopt}\label{thm-alpha-beta-subopt}
For $1\le\alpha<\beta\le2$, there is a constant $C_{\alpha,\beta}>0$ such that
    \begin{equation}\label{eq:main-res-beta}
        \epsilon_{\alpha,\beta}^{[m]} \le C_{\alpha,\beta}\,m^{3/2}e^{-\pi\sqrt{2\alpha(\beta-\alpha) m/\beta}}
    \end{equation}
    for all $m\ge1$.
    For $-1\le\beta<\alpha\le0$, we have
    \begin{equation}\label{eq:main-res-beta-transpose}
        \epsilon_{\alpha,\beta}^{[m]} \le C_{1-\alpha,1-\beta}\,m^{3/2}e^{-\pi\sqrt{2(1-\alpha)(\alpha-\beta) m/(1-\beta)}}.
        \end{equation}
\end{restatable}
We suspect that these bounds can be improved; for example, we believe that the right-hand side of \eqref{eq:main-res-beta-transpose} can be replaced by $C_{\alpha,\beta}e^{-2\pi\sqrt{\alpha(\beta-\alpha) m/\beta}}$. Note that this improvement is root exponential, since the factor of 2 in the exponent has moved outside of the square root. See Conjectures \ref{conj:alpha-beta-opt} and \ref{conj:stahl-ext} for more details.

We have made computer code used to calculate the quadrature rules $(u_i,t_i)_{i=1}^m$ and errors $\epsilon^{[m]}_{\alpha,\beta}$ available at \codeurl.
\end{itemize}

As an example of how our results are of practical relevance in numerical quantum information science, we offer (without proof) the following result based on \Cref{thm-alpha-beta-subopt} and the forthcoming work \cite{fawzi2023smaller}, showing that one can get efficient semidefinite approximations of the quantum relative entropy function. We denote by $\H^n$ the space of $n\times n$ Hermitian matrices, and by $\H^n_{++}$ the set of positive definite $n\times n$ Hermitian matrices.
\begin{theorem}
For any $m \geq 1$, there is a convex function $D^{[m]}(\rho\|\sigma)$ defined for $(\rho,\sigma) \in \H^n_{++} \times \H^n_{++}$ such that
    \begin{itemize}
        \item $D^{[m]}$ has an explicit semidefinite programming representation with $O(m)$ blocks of size $2n\times2n$ each,
        \item For any pair $(\rho,\sigma) \in \H^n_{++} \times \H^n_{++}$ such that $\tr \rho = \tr \sigma = 1$,
                \begin{equation}
                \label{eq:boundDDm}
                \left|D(\rho\| \sigma) - D^{[m]}(\rho\| \sigma)\right| \le \frac{\epsilon_{1,2}^{[m]}}{2}\left(\tr[\rho^2 \sigma^{-1}]-1\right)
                \end{equation}
                where $\epsilon_{1,2}^{[m]}=O(m^{3/2}e^{-\pi\sqrt{m}})$.
    \end{itemize}
\end{theorem}
The Gaussian quadrature-based approximations which have been used in previous works \cite{fawzi2019semidefinite,brown2021device,fawzi2022semidefinite,araujo2022quantum} have a much slower convergence with $m$, namely in $1/m^2$.
In fact if our \Cref{conj:alpha-beta-opt} is true (supported by the numerical evidence in Section \ref{sec:numerical}) then the bound \eqref{eq:boundDDm} is actually $\epsilon_{1,2}^{[m]}=O(e^{-\pi\sqrt{2m}})$.


\paragraph{Organization} Section \ref{sec:prelim:best} covers preliminaries concerning Gaussian quadrature and best rational approximations. Section \ref{sec:quad} deals with Gaussian quadrature approximations for operator monotone and convex functions and Section \ref{sec:best} deals with best rational approximants. Finally, Section \ref{sec:numerical} contains numerical illustrations of the results.


\section{Preliminaries}
\label{sec:prelim:best}

In this section, we review some important material concerning Gaussian quadrature, Pad\'e approximants, and best rational approximation theory.
Given nonnegative integers $m_1,m_2$, let $\cR_{m_1,m_2}$ denote the set of rational functions $r(x)=\frac{p(x)}{q(x)}$, where $p\in\R_{m_1}[x]$, $q\in\R_{m_2}[x]$ are polynomials with $\deg p\le m_1$ and $\deg q\le m_2$. We will sometimes call $\cR_{m_1,m_2}$ the set of rational functions of order $[m_1/m_2]$.

\subsection{Gaussian quadrature and Pad\'e approximants}\label{sec:prelim:gauss}

Let $\mu$ be a finite measure on $[0,1]$ which is not supported on a finite set of points. For each positive integer $m$, there is a quadrature rule on $m$ nodes (the $m$-node Gauss quadrature rule for $\mu$)  such that for each $k=0,1,\dots,2m-1$, 
\begin{equation}\label{eq:poly}
    \int_{[0,1]}t^k\d\mu(t) = \sum_{i=1}^m u_it_i^k.
\end{equation}
The $m$ nodes $t_i$ are precisely the roots of the degree-$m$ orthogonal polynomial with respect to the measure $\mu$. 
It will be convenient to use the notation $\mu_m:=\sum_i u_i \delta_{t_i}$ for the $m$-node Gauss quadrature rule for $\mu$.

Alternatively, one can fix one or more of the nodes in advance, and choose the weights and remaining nodes such that \eqref{eq:poly} is satisfied for $k$ as large as possible. This leads to quadrature rules such as the Gauss-Radau or Gauss-Lobatto rule defined next.

The \emph{Gauss-Radau} quadrature rule for $\mu$ fixes \emph{either} $t_1=0$ or $t_m=1$, and satisfies \eqref{eq:poly} for $k=0,\dots,2m-2$. The interior nodes are the roots of the degree-$(m-1)$ orthogonal polynomial with respect to the modified measure whose density with respect to $\mu$ is $t$ or $1-t$ (depending on whether the fixed node is 0 or 1). We will write $\mu_m^0$, $\mu_m^1$ for the corresponding discrete measures.

The \emph{Gauss-Lobatto} quadrature rule for $\mu$ fixes \emph{both} $t_1=0$ and $t_m=1$, and satisfies \eqref{eq:poly} for $k=0,\dots,2m-3$. The interior nodes are the roots of the degree-$(m-2)$ orthogonal polynomial with respect to the modified measure whose density with respect to $\mu$ is $t(1-t)$. We will write $\mu_m^{0,1}$ for the corresponding discrete measure.

Given a function $f$, smooth in a neighbourhood of $1$, and nonnegative integers $m_1,m_2$, the \emph{Pad\'e approximant} to $f$ at $1$ of order $[m_1/m_2]$ is the rational function $r(x)=\frac{p(x)}{q(x)}$ which satisfies
\[q(x)f(x) - p(x) = O((x-1)^{m_1+m_2+1})\qquad\text{as }x\to1.\]
With this definition, the Pad\'e approximant of each order $[m_1/m_2]$ always exists and is unique.
Usually, Pad\'e approximants satisfy the slightly stronger condition $r(x)-f(x)=O((x-1)^{m_1+m_2+1})$, but this is not always possible for certain functions $f$.

\subsection{Best rational approximations}\label{sec:pelim:best}

Given an interval $I\subseteq\R$, and a continuous function $f:I\to\R$ define the best rational approximation error
\[E_{m_1,m_2}(f, I) = \inf_{r\in\cR_{m_1,m_2}}\,\sup_{x\in I}\,\vert r(x) - f(x)\vert.\]
Bernstein \cite{bernstein1912} already showed that if $I$ is bounded and if $f$ can be analytically continued to an open interval strictly containing $I$, then $f$ can be approximated by \emph{polynomials} with a geometric rate of convergence, i.e.
\[E_{m,0}(f, I) = O(\rho^m)\quad\text{ for some }\rho\in(0,1).\]
Unfortunately, the best polynomial approximants to functions with endpoint singularities can converge much more slowly. For example \cite{bernstein1914}, we have $E_{m,0}(\sqrt{x}, [0,1])=\Omega(n^{-1})$. On the other hand, rational approximations can be much better, as shown by Stahl in \cite{stahl1993} (see also \cite{newman})
\[E_{m,m}(\sqrt{x}, [0,1])\sim8e^{-\pi\sqrt{2m}}.\]
This root-exponential convergence is typical for functions admitting special integral form, namely \emph{Stieltjes transforms} of well-behaved measures. Though distinct, these functions have a strong connection with operator monotone and convex functions, a connection that we exploit heavily in this paper.

The following result is easily deduced from Theorems 1 and 2 in \cite{Pekarskii1995} concerning the best rational approximation of Markov-Stieltjes functions. Similar results are obtained in \cite{Borwein1983,Andersson1988,Andersson1994,Mochalina2008}.

\begin{theorem}\label{thm:root-exp}
    For $\alpha,\beta>0$,  let $\phi:[-1,1]\to\R$ be a Borel-measurable function satisfying $0\le\phi(\lambda)\le c(1-\lambda)^\alpha(1+\lambda)^\beta$ for some $c>0$.
    Let $G:[-1,1]\to\R$ be given by 
    \begin{equation} \label{eq:markov-function}G(w) = \int_{-1}^1 \frac{\phi(\lambda)}{1-\lambda w}\d \lambda.
    \end{equation}
    Then, for some constant $C>0$ and all $m\ge0$,
    \[E_{m,m}(G, [-1,1]) \le Ce^{-2\pi\sqrt{\kappa m}}\]
    where $\kappa:=\frac{\alpha\beta}{\alpha+\beta}$ is the harmonic mean of $\alpha$ and $\beta$.
\end{theorem}
Note that the function $G$ of \eqref{eq:markov-function} has two singularities at $w=-1$ and $w=+1$; indeed its $\lceil \alpha \rceil$\textsuperscript{th} derivative blows up as $w\to1$, while its $\lceil \beta \rceil$\textsuperscript{th} derivative blows up as $w\to-1$.
\begin{proof}
This result appears in the literature \cite{Pekarskii1995} when the function $G(w)$ admits a single singularity at $w=+1$, which corresponds to the case ``$\beta=+\infty$'' above. To deal with functions admitting two singularities we split the integral representation \eqref{eq:markov-function} into three terms $G=G^{-} + G^0 + G^+$ where $G^0$ is analytic and $G^-$ and $G^+$ each have a single singularity at $-1$ and $+1$ respectively. Applying existing results to each individual function yields the desired result. The details are worked out in Appendix \ref{sec:proof-root-exp}.
\end{proof}

A natural question is whether the best rational approximants to the function $G$ in \eqref{eq:markov-function} can be obtained by discretizing the integral form. In fact, it can be shown to be the case, and this is the object of the next theorem from \cite{Braess12}.
 \begin{theorem}[{See \cite[Theorem V.3.6]{Braess12}}]
     \label{thm:best=quad}
     Let $G$ be a function of the form \eqref{eq:markov-function}.
      Then for each $m\in\N$, $E_{m-1,m}(G; [-1,1])$ is attained by a rational function $R_m$ which has the form
     \[R_m(w) = \sum_{i=1}^m \frac{a_i}{1-\lambda_i w},\]
     for $a_i\ge0$ and $\lambda_i\in(-1,1)$. Moreover $R_m$ is unique.
 \end{theorem}

It will later be convenient to express the above results in terms of functions defined on $(0,\infty)$. By applying a change of variables $x=\frac{1+w}{1-w}$ which maps $w \in (-1,1)$ to $x \in (0,\infty)$, the function $G(w)$ in \eqref{eq:markov-function} is mapped to
\[
g(x) = \int_{0}^{1} \frac{x+1}{1+t(x-1)} d\mu(t)
\]
where $\frac{d\mu(t)}{dt} \leq c t^{\alpha} (1-t)^{\beta}$. The integral above is closely related to the integral representation of operator monotone and operator convex functions we saw earlier; more precisely we see that $\frac{x-1}{x+1} g(x)$ has exactly the form \eqref{eq:intrepopmon} and, $\frac{(x-1)^2}{x+1} g(x)$ has exactly the form \eqref{eq:intrepopcvx}. This allows us to state the following corollary concerning quadrature approximations of certain operator convex functions.
\begin{cor}
    \label{cor:root-exp-op-mon-cvx}
    Let $f:(0,\infty)\to \RR$ be an operator convex function with $f(1)=f'(1) = 0$ that admits an integral representation 
    $
    f(x) = \int_{0}^{1} \frac{(x-1)^2}{1+t(x-1)} \psi(t) dt
    $,
    where the density $\psi$ satisfies $0 \leq \psi(t) \leq c t^{\alpha} (1-t)^{\beta}$ for some $c,\alpha,\beta > 0$. 
    Then there is a constant $C$, and for each $m\in\N$ there are weights $u_i \geq 0$ and nodes $t_i \in (0,1)$, such that 
    \[\left|f(x) - \sum_{i=1}^m u_i \frac{(x-1)^2}{1+t_i(x-1)}\right| \le C e^{-2\pi \sqrt{\frac{\alpha\beta m}{\alpha+\beta}}} \,\frac{(x-1)^2}{x+1}. \]
    \end{cor}
\begin{proof}[Proof of \Cref{cor:root-exp-op-mon-cvx}]
The proof is a direct consequence of Theorems \ref{thm:root-exp} and \ref{thm:best=quad}, via a change of variables.
Indeed, note that
\[
\frac{x+1}{(x-1)^2} f(x) = \int_{0}^{1} \frac{x+1}{1+t(x-1)} \psi(t) dt = \int_{-1}^{1} \frac{1}{1-\lambda \frac{x-1}{x+1}} \phi(\lambda) d\lambda = G((x-1)/(x+1))
\]
where $G(w) := \int_{-1}^{1} \frac{1}{1-\lambda w} \phi(\lambda) d\lambda$, and
 $\phi(\lambda) := \psi((1-\lambda)/2) \leq c' (1-\lambda)^{\alpha} (1+\lambda)^{\beta}$.
By Theorems \ref{thm:root-exp} and \ref{thm:best=quad} we know that the best order $[m-1/m]$ rational approximants to $G(w)$ have the form $R_m(w)=\sum_{i=1}^m\frac{a_i}{1-\lambda_i w}$, and satisfy 
\[|G - R_m| = E_{m-1,m}(G,[-1,1]) \leq E_{m-1,m-1}(G,[-1,1]) \leq C' e^{-2\pi \sqrt{\kappa (m-1)}}\leq C e^{-2\pi \sqrt{\kappa m}}\]
for all $m \geq 1$, where $\kappa=\frac{\alpha\beta}{\alpha+\beta}$. It follows that, defining $r_m(x) = R_m((x-1)/(x+1))$, we get
\[
\sup_{x \in (0,\infty)} \Bigl|\frac{x+1}{(x-1)^2} f(x) - r_m(x)\Bigr| = \sup_{w \in (-1,1)} |G(w) - R_m(w)| \leq C e^{-2\pi \sqrt{\kappa m}}.
\]
Note that $r_m(x)=\sum_{i=1}^m\frac{a_i}{1-\lambda_i (\frac{x-1}{x+1})}=\sum_{i=1}^mu_i\frac{x+1}{1+t_i(x-1)}$, where $t_i=(1-\lambda_i)/2$ and $u_i=a_i/2$.
\end{proof}

\paragraph{Weighted approximations} In this paper we will mostly deal with best rational approximations \emph{relative} to a nonnegative weight function $b:I\to \R_{\geq0}$. We define the relative approximation error by
\[
E_{m_1,m_2}(f, I; b) = \inf_{r\in\cR_{m_1,m_2}}\,\inf\left\{\epsilon\;:\; \vert r(x) - f(x)\vert \le \epsilon\,b(x)\; \forall x\in I \right\}.
\]
Note that Corollary \ref{cor:root-exp-op-mon-cvx} already says that for operator convex $f:(0,\infty)\to \RR$ such that $f(1)=f'(1)=0$ whose representing measure satisfies the bound $d\mu(t)/dt = \psi(t) \leq ct^{\alpha} (1-t)^{\beta}$, 
\[
E_{m+1,m}\left(f,(0,\infty);\frac{(x-1)^2}{x+1}\right) \leq C e^{-2\pi\sqrt{\alpha \beta m/(\alpha+\beta)}}.
\]

\section{Best local approximants}
\label{sec:quad}

In this section we prove Theorem \ref{thm:op-convex-pade}, the analogue of Theorem \ref{thm:op-mono-pade} for operator convex functions.
The proof of Theorem \ref{thm:op-mono-pade} itself is very similar, so we omit it.

\begin{theorem}\label{thm:op-convex-pade}
    Let $\mu$ be a finite Borel measure on $[0,1]$, and let $f_\mu(x)=\int_{0}^{1}\frac{(x-1)^2}{1+t(x-1)}\d\mu(t)$ be the corresponding operator convex function satisfying $f_{\mu}(1) = f'_{\mu}(1) = 0$. 
    Let $\mu_m$ and $\mu_m^{0,1}$ be the discrete measures associated to the $m$-node Gauss and Gauss-Lobatto quadrature rules for $\mu$, respectively.
    Let $f_{\mu_m}$ and $f_{\mu_m^{0,1}}$ be the rational operator convex approximations to $f_\mu$ arising from $\mu_m$ and $\mu_m^{0,1}$.
    Then
    \begin{enumerate}
        \item \label{ite:fir} For $m\ge0$,
        \[
        f_{\mu_m}(x)\le f_{\mu_{m+1}}(x)\le f_\mu(x) \leq f_{\mu_{m+1}^{0,1}}(x) \leq f_{\mu_m^{0,1}}(x)
        \]
        \item \label{ite:sec} Locally around $x=1$, we have
        \[
        \begin{aligned}
        f_{\mu_m}(x) - f_{\mu}(x) &= O((x-1)^{2m+2})\\
        f_{\mu_m^{0,1}}(x) - f_{\mu}(x) &= O((x-1)^{2m})
        \end{aligned}
        \]
        \item\label{ite:thrd} For any $x > 0$, $f_{\mu_{m+1}^{0,1}}(x)-f_{\mu_m}(x) \le \mu_{m+1}^{0,1}(\{0\})(x-1)^2+\mu_{m+1}^{0,1}(\{1\})\frac{(x-1)^2}{x}$.
    \end{enumerate}
\end{theorem}
\begin{proof}
    We will first prove \ref{ite:sec}, then \ref{ite:fir}, then \ref{ite:thrd}.
    \begin{enumerate}
        \item[\ref{ite:sec}.]    
        For $x\in(0,2)$ we can write
                \begin{align}
                    f_\mu(x) - f_{\mu_m}(x) &= \int\frac{(x-1)^2}{1+t(x-1)}\d\mu(t)- \int\frac{(x-1)^2}{1+t(x-1)}\d\mu_m(t)\nonumber\\
                    &=(x-1)^2\int\sum_{k=0}^\infty t^k(1-x)^k\d\mu(t)- (x-1)^2\int\sum_{k=0}^\infty t^k(1-x)^k\d\mu_m(t) && {\scriptstyle[\text{since } |x-1|<1]}\nonumber\\
                    &=(x-1)^{2m+2}\sum_{k=0}^\infty (1-x)^k\left[\int t^{k+2m}\d\mu(t) - \int t^{k+2m}\d\mu_m(t)\right] && {\scriptstyle[\text{using \eqref{eq:poly}}]}\label{eq:diffq2}\\
                    &=O((x-1)^{2m+2}).\nonumber
                \end{align}
        Similarly, 
                \begin{align}
                    f_{\mu}(x)-f_{\mu_m^{0,1}}(x)&=(x-1)^{2m}\sum_{k=0}^\infty(1-x)^k\left[\int t^{k+2m-2}\d\mu(t) - \int t^{k+2m-2}\d\mu^{0,1}_m(t)\right]\label{eq:diffq3}\\
                    &=O((x-1)^{2m}),\nonumber
                \end{align}
        since $m$-node Gauss-Lobatto quadrature is exact for polynomials up to degree $2m-3$.
        \item[\ref{ite:fir}.]
            We will prove that $f_{\mu_{m+1}}(x) \ge f_{\mu_m}(x)$ for each $m$ and $x>0$. Since $f_{\mu_m}$ converges to $f_{\mu}$ pointwise, this also proves that $f_{\mu_m}(x)\le f_\mu(x)$.
            Let $0<s_1^m<\dots<s_m^m<1$ be the nodes of $\mu_m$ and let $0<s_1^{m+1}<\dots<s_{m+1}^{m+1}<1$ be the nodes of $\mu_{m+1}$. Then
            \[ f_{\mu_{m+1}}(x) - f_{\mu_{m}}(x) = \frac{p(x)(x-1)^2}{\prod_{i=1}^{m}[1+s^m_i(x-1)]\prod_{j=1}^{m+1}[1+s^{m+1}_j(x-1)]}\]
            for some polynomial $p$ of degree at most $2m$. On the other hand, using \eqref{eq:diffq2} twice, we have
            \[f_{\mu_{m+1}}(x) - f_{\mu_{m}}(x) = \left[\int t^{2m}\d\mu(t) - \int t^{2m}\d\mu_m(t)\right](x-1)^{2m+2}+O((x-1)^{2m+3}).\]
            It follows that $p(x)=c(x-1)^{2m}$, where $c:=\int t^{2m}\d\mu(t) - \int t^{2m}\d\mu_m(t)$. 
            For any smooth function $f$, the residual $\int f(t)\d\mu(t) - \int f(t)\d\mu_m(t)$ has the same sign as $f^{(2m)}(\eta)$, for some $\eta\in(0,1)$ \cite[Eq. 8.4.10]{hildebrand}. In our case, $f^{(2m)}(\eta)= (2m)!>0$ for any $\eta$, hence $c>0$. Therefore, $f_{\mu_{m+1}}(x) \ge f_{\mu_m}(x)$ for each $m$ and $x>0$.
        
            Now, let $0<t_1^m<\dots<t_{m-2}^m<1$ be the internal nodes of $\mu^{0,1}_m$ and let $0<t_1^{m+1}<\dots<t_{m-1}^{m+1}<1$ be the internal nodes of $\mu^{0,1}_{m+1}$. Then
                \[ f_{\mu^{0,1}_{m+1}}(x) - f_{\mu^{0,1}_{m}}(x) = \frac{p(x)(x-1)^2}{x\prod_{i=1}^{m-2}[1+t^m_i(x-1)]\prod_{j=1}^{m-1}[1+t^{m+1}_j(x-1)]}\]
            for some polynomial $p$ of degree at most $2m-2$. From \eqref{eq:diffq3}, we deduce that $p(x)=\bar c\,(x-1)^{2m-2}$, where $\bar c:=\int t^{2m-2}\d\mu(t) - \int t^{2m-2}\d\mu^{0,1}_m(t)$.
            Unlike for Gaussian quadrature, for Gauss-\emph{Lobatto} quadrature, for any smooth function $f$, the residual $\int f(t)\d\mu(t) - \int f(t)\d\mu^{0,1}_m(t)$ has the \emph{opposite} sign to $f^{(2m-2)}(\eta)$, for some $\eta\in(0,1)$ \cite[Eq. 8.10.22]{hildebrand}. Hence, $\bar c<0$.
            Therefore $f_{\mu_2^{0,1}}(x) \ge f_{\mu_3^{0,1}}(x) \ge\dots\ge f_\mu(x)$ for $m\ge2$ and all $x>0$.  
        \item[\ref{ite:thrd}.]   Let $0<s_1<\dots<s_m<1$ be the nodes of $\mu_{m}$, and let $0<t_1<\dots<t_{m-1}<1$ be  
                the interior nodes of $\mu_{m+1}^{0,1}$. Also, write $u_0=\mu_{m+1}^{0,1}(\{0\})$ and $u_1=\mu_{m+1}^{0,1}(\{1\})$. We can write 
                \begin{align}
                    f_{\mu_{m+1}^{0,1}}(x)-f_{\mu_m}(x) &= (x-1)^2\left[u_0+\frac{u_1}x + \frac{p(x)}{\prod_{i=1}^{m-1}[1+t_i(x-1)]\prod_{j=1}^m[1+s_j(x-1)]}\right]\nonumber\\
                    &=(x-1)^2\left[\frac{u_0xQ(x) + u_1Q(x)+xp(x)}{xQ(x)}\right],\label{eq:form1q}
                \end{align}
                where $p$ is a polynomial of degree $2m-2$ and $Q(x)\equiv\prod_{i=1}^{m-1}[1+t_i(x-1)]\prod_{j=1}^m[1+s_j(x-1)]$.
                By item \ref{ite:sec}, $f_{\mu_{m+1}^{0,1}}(x)-f_\mu(x)=O((x-1)^{2m+2})$ and $f_{\mu_{m}}(x)-f_\mu(x)=O((x-1)^{2m+2})$, so $f_{\mu_{m+1}^{0,1}}(x)-f_{\mu_m}(x)=O((x-1)^{2m+2})$ as $x\to1$. Therefore, \[f_{\mu_{m+1}^{0,1}}(x)-f_{\mu_m}(x)=\frac{c(x-1)^{2m+2}}{xQ(x)},\] for some $c\ge0$. 
                
                Note that, for $x\ge1$, the function $x\to\frac{x-1}{1+t(x-1)}$ is increasing for every $t\in[0,1]$. Therefore, the product $\frac{c(x-1)^{2m}}{xQ(x)}$ is increasing on the semi-infinite interval $x\ge1$. By considering the term of leading order in \eqref{eq:form1q}, we see that $\frac{c(x-1)^{2m}}{xQ(x)}\to u_0$ as $x\to\infty$. Therefore, for every $x\ge1$, we have
                \[f_{\mu_{m+1}^{0,1}}(x)-f_{\mu_m}(x)\le u_0(x-1)^2.\]
                
                On the other hand, $Q(x)$ is increasing in $x$, so for every $x\in(0,1]$, $\frac{c(x-1)^{2m+2}}{xQ(x)}\le \frac{c(x-1)^{2m+2}}{xQ(0)} \le \frac{c(x-1)^{2}}{xQ(0)}$. Again comparing with \eqref{eq:form1q}, we see that $\frac{c}{Q(0)}=u_1$, so for every $x\in(0,1]$, we have
                \[f_{\mu_{m+1}^{0,1}}(x)-f_{\mu_m}(x)\le u_1\frac{(x-1)^2}{x}.\]

                It follows that for any $x>0$,
                \[f_{\mu_{m+1}^{0,1}}(x)-f_{\mu_m}(x) \le (x-1)^2\max\{u_0,\frac{u_1}{x}\} \le u_0(x-1)^2+u_1\frac{(x-1)^2}{x}.\]
    \end{enumerate}
\end{proof}

\begin{remark}
    $f_{\mu_m}(x)$ is a rational function of order $[m+1/m]$. Combining this with the second part of the theorem, $f_{\mu_m}(x)$ is the order $[m+1/m]$ Pad\'e approximant to $f_\mu(x)$ at $x=1$. 
    Also, $f_{\mu_m^{0,1}}(x)$ is a rational function of order $[m+1/m-1]$ with a simple pole at $x=0$. Therefore, $xf_{\mu_m^{0,1}}(x)$ is a rational function of order $[m+1/m-2]$. Multiplying by $x$, we have $xf_{\mu_m^{0,1}}(x)-xf_{\mu}(x)=O((x-1)^{2m})$ as $x\to1$. This shows that $xf_{\mu_m^{0,1}}(x)$ is the order $[m+1/m-2]$ Pad\'e approximant to $xf_\mu(x)$ at $x=1$.
\end{remark}
\begin{remark}
Some functions such as $x\mapsto \log x$ or $x\mapsto x^{\alpha}$ for $\alpha \in (0,1)$ are both operator monotone, and operator concave. One can thus apply either  \Cref{thm:op-mono-pade} or \Cref{thm:op-convex-pade} to obtain rational approximations. In general one obtains \emph{different} rational approximations. However the upper approximations from \Cref{thm:op-mono-pade} (based on Gauss-Radau with fixed node $t=0$) and \Cref{thm:op-convex-pade} (based on Gauss quadrature) coincide, since they are the $[m+1/m]$ Pad\'e approximants to these functions. Note that, since  $x\mapsto \log x$ and $x\mapsto x^{\alpha}$ for $\alpha \in (0,1)$ are operator \emph{concave} (not convex), the Gauss quadrature approximant from  \Cref{thm:op-convex-pade} is indeed an \emph{upper} approximant. The lower approximations from these theorems are different, however.
\end{remark}
\begin{remark}
    In practice, the $m$-node Gauss, Radau, and Lobatto quadrature rules can be obtained by solving an eigenvalue problem \cite{golubquad69,golubquadext73}. The matrix defining the eigenvalue problem has entries related to the measure $\mu$ (they are the coefficients for the recurrence relation satisfied by the orthogonal polynomials with respect to $\mu$).
\end{remark}

\subsection{The special case of $h(x)=\frac{x^\alpha-1}{\alpha(1-\alpha)}$}

In this section we consider the particular functions
\[
h_{\alpha}(x) = \frac{x^\alpha-1}{\alpha(1-\alpha)}\qquad \left[{}=\frac{x-1}{1-\alpha}-f_\alpha(x)\right]
\]
which are operator monotone for all $\alpha \in (-1,1)$. Note that $h_{\alpha}(x) \to \log x$ as $\alpha \to 0$. Since $h_{\alpha}$ is operator monotone, it has an integral representation which is explicitly given by
\begin{equation}
\label{eq:halpha}
h_{\alpha}(x) = \int_{0}^{1} \frac{x-1}{1+t(x-1)} d\nu_{\alpha}(t)
\end{equation}
where $d\nu_{\alpha}(t) = \frac{\sin(\alpha \pi)}{\alpha (1-\alpha)\pi} t^{-\alpha}(1-t)^{\alpha} dt$ (with $d\nu_0(t) = dt$).
The next theorem makes the convergence bound \eqref{eq:gaussianglobalerror} explicit in terms of $m$ and $\alpha$.
\begin{theorem}
Let $\alpha \in (-1,1)$, and let $\nu_{\alpha,m}^{0}$ and $\nu_{\alpha,m}^1$ be respectively the $m$-point Gauss-Radau discrete measures obtained from the integral representation \eqref{eq:halpha}. Then 
\[
\begin{aligned}
\nu_{\alpha,m}^0(\{0\}) &= \frac{\Gamma(1-\alpha)\Gamma(m+\alpha)}{\Gamma(1+\alpha)\Gamma(m+1-\alpha)\,m} \sim \frac{\Gamma(1-\alpha)}{\Gamma(1+\alpha)\,m^{2(1-\alpha)}}\\
\nu_{\alpha,m}^1(\{1\}) &= \nu_{-\alpha,m}^1(\{0\}) \sim \frac{\Gamma(1+\alpha)}{\Gamma(1-\alpha)\,m^{2(1+\alpha)}}.
\end{aligned}
\]
When $\alpha=0$, the asymptotic equalities are exact, i.e. $\nu_{\alpha,m}^0(\{0\})=\nu_{\alpha,m}^1(\{1\})=m^{-2}$.
\end{theorem}
\begin{proof}
This is an application of \cite[Eq. 3.10]{GAUTSCHI2000403} where it is shown that for the measure on $[-1,1]$ with density $(1-\lambda)^\alpha(1+\lambda)^\beta$, the $m$-point Gauss-Radau quadrature rule has weight
    \[u^{\alpha,\beta}_m:=\frac{2^{\alpha+\beta+1}\Gamma(1+\beta)\Gamma(2+\beta)\Gamma(m+\alpha)\Gamma(m)}{\Gamma(m+\beta+1)\Gamma(m+\alpha+\beta+1)}\]
    at the endpoint $\lambda=-1$.
    Note that the pushforward of $\nu_\alpha$ by $\lambda(t)=2t-1$ has density $\frac{\sin(\alpha \pi)}{2\alpha(1-\alpha) \pi}(1-\lambda)^\alpha(1+\lambda)^{-\alpha}$. Therefore, 
    \begin{align*}
        \nu_{\alpha,m}^0(\{0\}) &= \frac{\sin(\alpha \pi)}{2\alpha(1-\alpha) \pi}\cdot u^{\alpha,-\alpha}_m\\
        &=\frac{\sin(\alpha \pi)}{\alpha(1-\alpha) \pi}\cdot \frac{\Gamma(1-\alpha)\Gamma(2-\alpha)\Gamma(m+\alpha)\Gamma(m)}{\Gamma(m-\alpha+1)\Gamma(m+1)}\\
        &=\frac{1}{(1-\alpha)\Gamma(1-\alpha)\Gamma(1+\alpha)}\cdot \frac{\Gamma(1-\alpha)\Gamma(2-\alpha)\Gamma(m+\alpha)}{\Gamma(m-\alpha+1)m}\\
        &=\frac{\Gamma(1-\alpha)}{\Gamma(1+\alpha)}\cdot \frac{\Gamma(m+\alpha)}{\Gamma(m-\alpha+1)m}.
    \end{align*}
\end{proof}

By Stirling's formula, $\Gamma(x+\gamma)\sim\Gamma(x)x^\gamma$ as $x\to\infty$, so we obtain the asymptote
    \[\nu_{\alpha,m}^{0}(\{0\})\sim\frac{\Gamma(1-\alpha)}{\Gamma(1+\alpha)\,m^{2(1-\alpha)}}\]
    as $m\to\infty$.

    Finally, since $\nu_\alpha$ has density proportional to $t^{-\alpha}(1-t)^\alpha$, the pushforward of $\nu_\alpha$ by $t\mapsto1-t$ is $\nu_{-\alpha}$. It follows that
    \[\nu_{\alpha,m}^{1}(\{1\})=\nu_{-\alpha,m}^{0}(\{0\})\sim \frac{\Gamma(1+\alpha)}{\Gamma(1-\alpha)\,m^{2(1+\alpha)}}.\]

\section{Best global approximants of $\alpha$-divergences}
\label{sec:best}

In this section we study best global rational approximants. We focus in particular on the functions
\begin{equation}
\label{eq:falpha}
f_{\alpha}(x) = \frac{x^{\alpha}-\alpha(x-1)-1}{\alpha(\alpha-1)}
\end{equation}
for $\alpha \in [-1,2]$, which generate the so-called $\alpha$-divergences.
We note that $f_{\alpha}$ is operator convex on $(0,\infty)$, and that $f_{\alpha}(x) \geq 0$ for all $x > 0$ with $f(1) = f'(1) = 0$ and $f''(1) = 1$. As such, $f_{\alpha}$ has an integral representation of the form
\[
f_{\alpha}(x) = \int_{0}^{1} \frac{(x-1)^2}{1+t(x-1)} d\mu_{\alpha}(t)
\]
with 
\begin{equation}
\label{eq:falpha-repr-measure}
\frac{d\mu_{\alpha}}{dt} = \frac{\sin[(\alpha-1)\pi]}{\alpha (\alpha-1)} t^{1-\alpha} (1-t)^{\alpha}
\end{equation}
Note that for $\alpha \in (0,1)$, the asymptotes of $f_\alpha$  as $x\to0^+$ and as $x\to\infty$ are integral powers of $x$. Indeed, we have $f_\alpha(x)\to \frac1\alpha$ as $x\to0^+$ and $f_\alpha(x)\sim \frac{x}{1-\alpha}$ as $x\to\infty$. This is not the case for $\alpha\in(-1,0]\cup[1,2)$.
As such, the approximation results we have for $f_{\alpha}$ will differ depending on whether $\alpha \in (0,1)$ or not. Recall from Section \ref{sec:prelim:best} that $E_{m_1,m_2}(f,I)$ is the smallest error in approximating $f$ by a $[m_1/m_2]$ rational function on $I \subset \RR$, and that when $b \geq 0$ on $I$, $E_{m_1,m_2}(f,I;b)$ is the smallest error relative to $b$ in approximating $f$ by a $[m_1/m_2]$ rational function on $I$.

First we show that it is not possible to find a uniform rational approximation to $f_\alpha$ on the infinite interval $(0,\infty)$  for any $\alpha \in (-1,2)$.
\begin{theorem}
For any $\alpha \in (-1,2)$, $E_{m_1,m_2}(f_{\alpha},(0,\infty)) = \infty$ for any $m_1,m_2$. 
\end{theorem}
\begin{proof}
    For $\alpha\notin\{0,1\}$, it suffices to show that $E_{m_1,m_2}(x^\alpha,(0,\infty)) = \infty$.
    \\
    Let $r(x)$ be a rational function, and oberve that there is a number $c\in\R$ and an integer $k$ such that $r(x)\sim cx^k$ as $x\to\infty$. There is also a number $\bar c$ and integer $\bar k$ such that $r(x)\sim \bar cx^{\bar k}$ as $x\to0^+$.
    We have \[x^\alpha-r(x)\sim \begin{cases} x^\alpha(1-cx^{k-\alpha}) & x\to\infty\\ x^\alpha(1-\bar cx^{\bar k-\alpha})&x\to0^+.\end{cases}\] 
    If $\alpha\in(-1,0)$, then $x^\alpha\to\infty$ as $x\to0^+$, so $x^\alpha-r(x)\to\infty$ unless $\bar cx^{\bar k - \alpha}\to 1$. This is impossible, since $\bar k\neq\alpha$. 
    If $\alpha\in(0,1)\cup(1,2)$, then $x^\alpha\to\infty$ as $x\to\infty$, so $x^\alpha-r(x)\to\infty$ unless $cx^{k - \alpha}\to 1$. This is impossible, since $k\neq\alpha$. \\
    Finally, if $\alpha\in\{0,1\}$ essentially the same argument works, with $x^\alpha$ replaced by $\log x$ or $x\log x$.
\end{proof}


We now turn our attention to rational approximations \emph{relative} to a weight function $b$. Our first theorem shows that under some mild conditions on $f$ and $b$, the best rational approximant is of quadrature type, i.e., can be obtained by discretizing the integral representation of $f$. The next theorem is general, and is not restricted to the functions $f_{\alpha}$.
 
\begin{theorem}\label{thm:op-convex-quad}
    Let $f:(0,\infty)\to\R$ be operator convex with $f(1)=f'(1)=0$. Let $b:(0,\infty)\to\R$ be continuous, and positive except at $x=1$. Assume that 
    \begin{equation}\label{eq:h-cont-mid}
        \lim_{x\to1}b(x)/(x-1)^2>0,
    \end{equation}
    that the limits 
    \begin{align}
        \lim_{x\to0^+}f(x)/b(x),\quad &\lim_{x\to\infty}f(x)/b(x)\label{eq:fh-cont-ends}\\ \lim_{x\to0^+}1/b(x),\quad & \lim_{x\to\infty}x/b(x)\label{eq:h-cont-ends}
    \end{align} 
    also exist and are finite, and that
    \begin{equation}\label{eq:h-slowly-vanish}
        \lim_{x\to0^+}xb(x)=0,\quad \lim_{x\to\infty}b(x)/x^2=0.
    \end{equation}
    A best order $[m+1/m]$ rational approximation to $f$ relative to $b$ exists.
    Moreover, if $\tilde{f}$ is such a best approximation, then it has the form
    \begin{equation}\label{eq:f-def-op-quad}
        \tilde{f}(x)=\sum_{i=1}^m \frac{u_i(x-1)^2}{1+t_i(x-1)}
    \end{equation}
    for weights $u_i\ge0$ and nodes $0\le t_1<\dots < t_m \le 1$.
\end{theorem}
\begin{proof}
    See \Cref{sec:proof-osc}.
\end{proof}

\paragraph{Best rational approximants of $\alpha$-divergences} Recall from \eqref{eq:def-error-alpha} that $\epsilon^{[m]}_{\alpha,\beta}$ is the error of the best approximation of quadrature type to $f_{\alpha}$ relative to the weight function $f_{\beta}$, i.e.
\[\epsilon^{[m]}_{\alpha,\beta}:= \inf_{\substack{0\le t_1<\dots<t_m\le 1\\ u_i\ge0,\,i=1,\dots,m}} \left\{\,\sup_{x>0}\,
\Bigl|\frac{f_\alpha(x) - \tilde{f}(x)}{f_\beta(x)}\Bigr| \; :\; \tilde{f}(x)=\sum_{i=1}^m \frac{u_i(x-1)^2}{1+t_i(x-1)}\right\}.\]
As a direct corollary to \Cref{thm:weighted-osc}, we have
\begin{cor}\label{cor:alpha-approx-is-quad}
    Suppose that either $\alpha,\beta\in(0,1)$, or $1\le\alpha<\beta<2$, or $-1<\beta<\alpha\le0$. Then 
    \[\epsilon^{[m]}_{\alpha,\beta}=E_{m+1,m}(f_\alpha, (0,\infty); f_\beta).\]
\end{cor}

Our main theorems in this section concern the rate of decay of $\epsilon^{[m]}_{\alpha,\beta}$. Our first theorem deals with the case $\alpha \in (0,1)$. The theorem will be proved in \Cref{sec:01proof}.
\thmapproxalphaopt*
For $\alpha\in(-1,0]\cup[1,2)$, it turns out that $f_\alpha$ cannot be well approximated by rational functions $\tilde{f}$ in relative error:
\begin{theorem}
For $\alpha\in(-1,0]\cup[1,2)$, we have $\epsilon_{\alpha,\alpha}^{[m]}=1$.
\end{theorem}
\begin{proof}
    First consider $\alpha\in(-1,0]$, and consider a rational approximation $\tilde{f}$ defined by nodes $0\le t_1<\dots<t_m\le 1$ and weights $u_i\ge0$. If $t_m=1$ and $u_m>0$, we have $\lim_{x\to0^+}x\tilde{f}(x)=u_m>0$, and since $\lim_{x\to0^+}xf_\alpha(x)=\infty$,  $\sup_{x>0}\,
    \bigl|\frac{\tilde{f}(x)}{f_\alpha(x)}-1\bigr|=\infty$.
    Otherwise (if $t_m<1$ or $u_m=0$), then $\lim_{x\to0^+}\tilde{f}(x)=\sum_i\frac{u_i}{1-t_i}<\infty$, and since $\lim_{x\to0^+}f_\alpha(x)=\infty$, we have $\sup_{x>0}\,
    \bigl|\frac{\tilde{f}(x)}{f_\alpha(x)}-1\bigr|\ge1$.
    By setting all the weights are zero (so that $\tilde{f}(x)\equiv0$), we can always obtain $\sup_{x>0}\,
    \bigl|\frac{\tilde{f}(x)}{f_\alpha(x)}-1\bigr|=1$, so $\epsilon_{\alpha}^{[m]}=1$.

    For $\alpha\in[1,2)$, analogous considerations of the behaviour of $\tilde{f}(x)$ and $f_\alpha(x)$ as $x\to\infty$ show that $\epsilon_{\alpha}^{[m]}=1$ in this case as well.
\end{proof}

This suggests to consider approximations which minimise the error $\epsilon^{[m]}_{\alpha,\beta}$ for $\beta\neq\alpha$.
The following will be proved in \Cref{sec:12proof} by explicitly constructing (suboptimal) quadrature rules.
\thmalphabetasubopt*
Numerical experiments (see \Cref{fig:alpha} -- right) suggest that the rate of root exponential convergence in \Cref{thm-alpha-beta-subopt} is too pessimistic by a factor of $\sqrt2$. Further evidence of the suboptimality of this result is that, when modified to provide a bound on $\epsilon_{\alpha,\alpha}^{[m]}$ for $\alpha\in(0,1)$, the technique used to prove \Cref{thm-alpha-beta-subopt} yields only a bound of the form
\[\epsilon_{\alpha,\alpha}^{[m]} \le C_\alpha \,m^{3/2} e^{-\pi\sqrt{2\alpha(1-\alpha)m}},\]
but we know from \Cref{thm:approx-alpha-opt} that the correct behabiour is $e^{-2\pi\sqrt{\alpha(1-\alpha)m}}$.
Therefore we conjecture the following
\begin{conjecture}\label{conj:alpha-beta-opt}
    In \Cref{thm-alpha-beta-subopt}, the right-hand-side of \eqref{eq:main-res-beta} can be replaced by
        \[ C_{\alpha,\beta}e^{-2\pi\sqrt{\alpha(\beta-\alpha) m/\beta}},\]
    and the right-hand side of \eqref{eq:main-res-beta-transpose} can be replaced by
    \[ C_{1-\alpha,1-\beta}e^{-2\pi\sqrt{(1-\alpha)(\alpha-\beta) m/(1-\beta)}}.\]
\end{conjecture}

We note that to prove the conjecture above, it would be sufficient to prove the following result:
    \begin{conjecture}\label{conj:stahl-ext}
        Let $-1\le\beta<\alpha<0$. There is a constant $C$ such that
        \[E_{m,m}(x^\alpha, (0,1]; x^\beta)\le C\,e^{-2\pi\sqrt{(\alpha-\beta)m}}. \]
    \end{conjecture}
    The above can be seen as an extension to negative powers $\alpha$ of the following famous estimate in approximation theory \cite{stahl1993}:
    \[
    \forall \alpha > 0, \quad E_{m,m}(x^\alpha, [0,1])\sim 4^{1+\alpha}|\sin[\alpha\pi]|\,e^{-2\pi\sqrt{\alpha m}}.
    \]

\subsection{Proof of Theorem \ref{thm:approx-alpha-opt} (case $\alpha \in (0,1)$)}\label{sec:01proof}



Since $\alpha,1-\alpha > 0$, we can readily apply Corollary \ref{cor:root-exp-op-mon-cvx} which says that for any $m\in\N$ we can find weights $u_i\geq0$ and nodes $t_i\in(0,1)$ such that 
\begin{equation}
\left|f_\alpha(x) - \sum_{i=1}^m\frac{u_i\,(x-1)^2}{1+t_i(x-1)}\right| \leq C e^{-2\pi \sqrt{\alpha(1-\alpha)m}}\cdot\frac{(x-1)^2}{x+1} \quad \forall x > 0.
\end{equation}
The key insight is to observe that the function 
\[
\frac{(x+1)}{(x-1)^2} f_{\alpha}(x)
\]
is  bounded below by a strictly positive constant, in fact by $1/3$ (see below). 
This implies
\[
\frac{|f_{\alpha}(x) - \tilde{f}(x)|}{f_{\alpha}(x)} \leq 3C e^{-2\pi \sqrt{\alpha(1-\alpha)m}}
\]
where $\tilde{f}(x) = \sum_{i=1}^m\frac{u_i\,(x-1)^2}{1+t_i(x-1)}$ has the required form.

It remains to prove that $\frac{(x+1)}{(x-1)^2} f_{\alpha}(x)\ge\frac13$ for all $x>0$.
This is a special case of \Cref{lem:helper-bound-rational}.

\subsection{Proof of \Cref{thm-alpha-beta-subopt} (case $\alpha\in(-1,0]\cup[1,2)$)}\label{sec:12proof}
We start by making a change of variables in the integral representation of $f_\alpha(x)$, so that the integral is over $\R$ instead of $(0,1)$:
\begin{align}
    f_\alpha(x) &= \frac{(x-1)^2}{Z_\alpha}\int_0^1 \frac{t^{1-\alpha}(1-t)^\alpha}{1+t(x-1)}\d t &&\text{where }Z_\alpha:=\frac{\alpha(\alpha-1)\pi}{\sin[(\alpha-1)\pi]}\nonumber\\
    &= \frac{(x-1)^2}{Z_\alpha}\int_{-\infty}^\infty \frac{e^u(\frac{1}{1+e^{-u}})^{1-\alpha}(\frac{e^{-u}}{1+e^{-u}})^\alpha}{(1+e^u)(1+xe^u)}\d u&& \text{change of variable }t=\frac{1}{1+e^{-u}}\nonumber\\
    &=\frac{(x-1)^2}{Z_\alpha}
    \int_{-\infty}^\infty \frac{e^{(2-\alpha)u}}{(1+e^u)^2(1+xe^u)}\d u.\label{eq:exp-int}
\end{align}
An outline of the construction is as follows. We approximate the integral \eqref{eq:exp-int} using the trapezoidal rule at $m$ equispaced nodes a distance $h$ apart. The total error of this approximation is the sum of the discretization error (going from the integral to an infinite sum) and the truncation error (from truncating the sum to $m$ terms). Relative to the function $f_\beta(x)$, the discretization error is of order $h^{-3}e^{-2\pi^2/h}$ (\Cref{lem:discretisation}). The truncation error is of order $e^{(\beta-\alpha)\alpha m h/\beta}$ (\Cref{lem:bound-trunc}). The exponential part of these estimates are approximately balanced when $h=\pi\sqrt{\frac{2\beta}{\alpha(\beta-\alpha)m}}$, and the overall error is then of order $m^{3/2}e^{-\pi\sqrt{2\alpha(\beta-\alpha)m/\beta}}$ as claimed.

The fundamental approach of this construction is due to Stenger \cite{stenger}; Trefethen gives a simplified exposition in \cite[Chapter 25]{approx-practice}. Our analysis is slightly different because we are interested in best uniform rational approximations \emph{relative} to a function $f_\beta(x)$. The key technical result is \Cref{lem:helper-bound-rational-comp}, which  is used in \Cref{lem:bound-trunc} to bound the truncation error relative to $f_\beta(x)$.

We continue with a more detailed presentation of the construction. The integral in \eqref{eq:exp-int} can be approximated by the discrete sum
\begin{equation}\label{eq:def-s-h}
S_\alpha^h(x):=\frac{{(x-1)^2}}{Z_\alpha}\sum_{n=-\infty}^\infty \frac{he^{(2-\alpha)nh}}{(1+e^{nh})^2(1+xe^{nh})},
\end{equation}
for some small $h>0$,
which can in turn be truncated to a sum 
\begin{equation}\label{eq:s-h-m1-m2}
    S_\alpha^{h,m_-,m_+}(x):=\frac{(x-1)^2}{Z_\alpha}\sum_{n=m_-}^{m_+} \frac{he^{(2-\alpha)nh}}{(1+e^{nh})^2(1+xe^{nh})}
\end{equation}
with $m=m_+-m_-+1$ terms, where $m_-\le0\le m_+$. Note that $S_\alpha^{h,m_-,m_+}(x)$ has the form $\sum_{n} u_n\cdot\frac{{(x-1)^2}}{1+t_n(x-1)}$, where $t_n=(1+e^{-nh})^{-1}$ and $u_n=\frac{he^{(2-\alpha)nh}}{Z_\alpha(1+e^{nh})^3}$.

To complete the proof of \Cref{thm-alpha-beta-subopt}, we will need two main lemmas. These will also guide our choice of the parameters $h$ (the discretisation scale), and $m_-,m_+$ (which determine the truncation of \eqref{eq:def-s-h}).

\begin{lemma}\label{lem:discretisation}
    There is an explicit absolute constant $C>0$ such that for every $\alpha\in[-1,2]$ and $h<\frac{\pi^2}2$,
    \begin{equation}\label{eq:discr-content}
    \left| f_\alpha(x) - S_\alpha^h(x)\right|\le Ch^{-3}e^\frac{-2\pi^2}{h}f_\alpha(x)\qquad\forall\,x>0.
    \end{equation}
    Here  $S_\alpha^h$ is as defined in \eqref{eq:def-s-h}.
\end{lemma}

\begin{lemma}\label{lem:bound-trunc}
    Let $1\le\alpha<\beta\le2$, $h>0$, and $m_-\le0\le m_+$. Then, for each $x>0$, 
    \[0 \le S_\alpha^h(x) - S_\alpha^{h,m_-,m_+}(x) \le \frac{3}{Z_\alpha}\left(\frac{e^{(\beta-\alpha)hm_-}}{\beta-\alpha} + \frac{e^{-\alpha h m_+ }}{\alpha}\right)f_\beta(x).\]
\end{lemma}

\begin{proof}
    [Proof of \Cref{thm-alpha-beta-subopt}]
    Assume first that $1\le\alpha<\beta\le2$.

    Let $m_+$ be the largest integer which is less than $(1-\frac\alpha\beta)m$, and let $m_-$ be the smallest integer greater than $-\frac\alpha\beta m$. Then the sum $S_\alpha^{h,m_-,m_+}(x)$ has at most $m$ terms, and by \Cref{lem:bound-trunc},
    \begin{align*}
        |S_\alpha^h(x) - S_\alpha^{h,m_-,m_+}(x)| &\le \frac{3}{Z_\alpha}\left(\frac{e^{(\beta-\alpha)hm_-}}{\beta-\alpha} + \frac{e^{-\alpha h m_+ }}{\alpha}\right)f_\beta(x)\\
        & \le \frac{3}{Z_\alpha}\left(\frac{e^{-(\beta-\alpha)h[\frac\alpha\beta m - 1]}}{\beta-\alpha} + \frac{e^{-\alpha h [(1-\frac\alpha\beta)m-1] }}{\alpha}\right)f_\beta(x)\\
        &=\frac{3}{Z_\alpha}\left(\frac{e^{(\beta-\alpha)h}}{\beta-\alpha} + \frac{e^{\alpha h  }}{\alpha}\right)e^{-\frac{\alpha(\beta-\alpha)}{\beta}mh}f_\beta(x).
    \end{align*}

    By \Cref{lem:discretisation}, $\left| f_\alpha(x) - S_\alpha^h(x)\right|\le Ch^{-3}e^\frac{-2\pi^2}{h}f_\alpha(x)$. Combining this with the fact that $f_\alpha(x)\le \frac\beta\alpha f_\beta(x)$ (see \Cref{lemma:f-incr-in-alpha}), we get
    \[\left| f_\alpha(x) - S_\alpha^h(x)\right|\le {\frac{C\beta}{\alpha}h^{-3}e^\frac{-2\pi^2}{h}f_\beta(x)}.\]

    Choosing $h=\pi\sqrt{\frac{2\beta}{\alpha(\beta-\alpha)m}}$ (so that $e^{-\frac{\alpha(\beta-\alpha)}{\beta}mh}=e^\frac{-2\pi^2}{h}$), we combine our estimates to obtain the bound
    \[|f_\alpha - S_\alpha^{h,m_-,m_+}(x)| \le \underbrace{\left(\frac{C\beta}{\alpha}h^{-3} + \frac{3}{Z_\alpha}\left(\frac{e^{(\beta-\alpha)h}}{\beta-\alpha} + \frac{e^{\alpha h  }}{\alpha}\right)\right)}_{=O(m^\frac{3}{2})\text{ since }h=\sqrt{\frac{2\beta}{\alpha(\beta-\alpha)m}}}e^{-\pi\sqrt{2\alpha(\beta-\alpha)m/\beta}}f_\beta(x).\]
    This proves \eqref{eq:main-res-beta}.

    In the case where $-1\le\beta<\alpha\le0$, \eqref{eq:main-res-beta-transpose} follows immediately by observing that $\epsilon_{\alpha,\beta}^{[m]}=\epsilon_{1-\alpha,1-\beta}^{[m]}$. This is because $f_{1-\alpha}(x) \equiv xf_\alpha(\frac1x)$.
\end{proof}

To prove Lemma \ref{lem:discretisation}, we will use the following result about the accuracy of the trapezoidal rule for analytic integrands.
\begin{theorem}[Theorem 5.1 in \cite{trefethen-trap}]
    Let $\omega$ be a function analytic in the strip $|\Im(z)|<a$, and such that $\omega(z)\to0$ uniformly as $|z|\to\infty$ in the strip. Suppose further that for some $M>0$,
    \begin{equation}\label{eq:M-bound}
    \int_{-\infty}^\infty |\omega(u+bi)| \d u \le M
    \end{equation}
    for every $b\in(-a,a)$. Then for any $h>0$,
    \[\left| \int_{-\infty}^\infty \omega(u)\d u - h\sum_{j=-\infty}^\infty \omega(jh)  \right| \le \frac{2M}{e^{2\pi a/h}-1}.\]
\end{theorem}

\begin{proof}
    [Proof of Lemma \ref{lem:discretisation}]
    Define the functions 
    \begin{equation}
        \label{eq:def-omega-x}
        \omega_x(u) := \frac{e^{(2-\alpha)u}}{(1+e^u)^2(1+xe^u)}.
    \end{equation}
    Note that \eqref{eq:discr-content} is equivalent to 
    \[\left| \int_{-\infty}^\infty \omega_x(u)\d u - h\sum_{j=-\infty}^\infty \omega_x(jh)  \right| \le Ch^{-3}e^\frac{-2\pi^2}{h}\int_{-\infty}^\infty \omega_x(u)\d u.\]
    
    For each $x$, $\omega_x$ is analytic in the strip $|\Im(z)|<\pi$, and $\omega_x(z)\to0$ uniformly as $|z|\to\infty$ in the strip. However there is no finite $M$ satisfying \eqref{eq:M-bound} for every $|b|<\pi$.
    On the other hand, by Lemma \ref{lem:im-bound}, for $\epsilon\in(0,1)$ we have 
    \[\int_{-\infty}^\infty |\omega(u+bi)| \d u \le \cos[(1-\epsilon)\pi/2]^{-3}\int_{-\infty}^\infty \omega_x(u)\d u \le \frac1{\epsilon^3}\int_{-\infty}^\infty \omega_x(u)\d u\]
    whenever $b\in((\epsilon-1)\pi, (1-\epsilon)\pi)$.

    Consequently, for each $\epsilon\in(0,1)$ we have 
    \[\left| \int_{-\infty}^\infty \omega_x(u)\d u - h\sum_{j=-\infty}^\infty \omega_x(jh)  \right| \le \frac{2}{\epsilon^3[e^{2\pi^2(1-\epsilon)/h}-1]}\int_{-\infty}^\infty \omega_x(u)\d u.\]
    
    We are free to choose $\epsilon=\frac{3h}{2\pi^2}<1$ (this value is chosen to maximise $\epsilon^3 e^{2\pi^2(1-\epsilon)/h}$). 
    With this choice of $\epsilon$ we have $2\pi^2(1-\epsilon)/h=2\pi^2/h-3 >1$ (since $h<\frac{\pi^2}{2}$). Therefore $e^{2\pi^2(1-\epsilon)/h} > e > 2$, and consequently, $2(e^{2\pi^2(1-\epsilon)/h}-1)>e^{2\pi^2(1-\epsilon)/h}$.
    
    We can now write
    \[\left| \int_{-\infty}^\infty \omega_x(u)\d u - h\sum_{j=-\infty}^\infty \omega_x(jh)  \right| < \frac{4}{\epsilon^3 e^{2\pi^2(1-\epsilon)/h}} = 4\left(\frac{2\pi^2e}{3h}\right)^3e^{-\frac{2\pi^2}{h}}.\]
    Lemma \ref{lem:discretisation} now follows, and we can take $C=4\left(\frac{2\pi^2e}{3}\right)^3$.
\end{proof}

\begin{proof}[Proof of \Cref{lem:bound-trunc}]
    We have 
    \begin{align*}
        \frac{Z_\alpha}{(x-1)^2}\left[S_\alpha^h(x) - S_\alpha^{h,m}(x)\right] &= 
        \sum_{n<m_-} \frac{he^{(2-\alpha)nh}}{(1+e^{nh})^2(1+xe^{nh})} + \sum_{n>m_+}\frac{he^{(2-\alpha)nh}}{(1+e^{nh})^2(1+xe^{nh})}\\
        &=\sum_{n<m_-} \frac{he^{(2-\alpha)nh}}{(1+e^{nh})^2(1+xe^{nh})} + \sum_{n>m_+}\frac{he^{-\alpha nh}}{(e^{-nh}+1)^2(1+xe^{nh})}\\
        & \le \sum_{n<m_-} \frac{he^{(2-\alpha)nh}}{1+xe^{nh}} + \sum_{n>m_+}\frac{he^{-\alpha nh}}{1+xe^{nh}}.
    \end{align*}
    By Lemma \ref{lem:helper-bound-rational-comp}, for $n>0$ we have $\frac1{3}\cdot\frac{(x-1)^2}{1+xe^{nh}}\le f_\beta(x)$, while for $n<0$ we have $\frac{(x-1)^2}{3}\cdot\frac{e^{(2-\beta)nh}}{1+xe^{nh}}\le f_\beta(x)$. 
    Therefore 
    \begin{align*}
        Z_\alpha\left[S_\alpha^h(x) - S_\alpha^{h,m_-,m_+}(x)\right] &\le 
        3h\left(\sum_{n<m_-} e^{(\beta-\alpha)nh} + \sum_{n>m_+}e^{-\alpha nh}\right)f_\beta(x)\\
         &=
        3h\left(\sum_{n>-m_-} e^{-(\beta-\alpha)nh} + \sum_{n>m_+}e^{-\alpha nh}\right)f_\beta(x)\\
        &=
        3h\left(\frac{e^{(\beta-\alpha)hm_- }}{e^{(\beta-\alpha)h}-1} + \frac{e^{-\alpha h m_+ }}{e^{\alpha h}-1}\right)f_\beta(x)\\
        &\le
        3\left(\frac{e^{(\beta-\alpha)hm_-}}{\beta-\alpha} + \frac{e^{-\alpha h m_+ }}{\alpha}\right)f_\beta(x).
    \end{align*}
\end{proof}

\section{Numerical illustration}
\label{sec:numerical}

In this section we numerically validate the convergence results obtained in this paper. To compute the best global approximants $\tilde{f}$ from Section \ref{sec:best}, we used the differential correction algorithm which we briefly review below in Section \ref{sec:computing-best}. In \Cref{fig:alpha} (left) we illustrate \Cref{thm:approx-alpha-opt}, by comparing the exact relative approximation errors $\epsilon_{\alpha,\alpha}^{[m]}$ of the best quadrature-based approximation of $f_\alpha$ with the predicted asymptotic error (according to \Cref{thm:approx-alpha-opt}). We illustrate the case where $\alpha\ge\frac{1}{2}$ only, since $\epsilon^{[m]}_{\alpha,\alpha}=\epsilon^{[m]}_{1-\alpha,1-\alpha}$. In \Cref{fig:alpha} (right) we consider the approximation errors $\epsilon_{1,\beta}^{[m]}$ for $\beta\in(1,2]$, which are relevant in the approximation of quantum relative entropy. These values are compared visually with the asymptotic errors predicted by \Cref{conj:alpha-beta-opt}. The obtained approximation errors match very well with the predicted rates of convergence.

In \Cref{fig:osc} we plot the error of the best $15$-node approximation to $f_1$ relative to $f_2$, and we observe the expected equioscillation. In \Cref{fig:nodes}, we compare the location of the nodes of this approximation with the nodes of the best \emph{local} approximation around $x=1$ (obtained by Gaussian quadrature on $[0,1]$ of the measure $\d\mu_\alpha(t) = (1-t)\d t$, see \Cref{sec:quad}). In fact, the nodes of the best local approximation are none other than the roots of the Jacobi polynomial $J^{0,1}_{15}$ after a linear transformation of the domain $[-1,1]$ to $[0,1]$. These roots can be obtained from the eigendecomposition of a certain tridiagonal matrix \cite{golubquad69}.

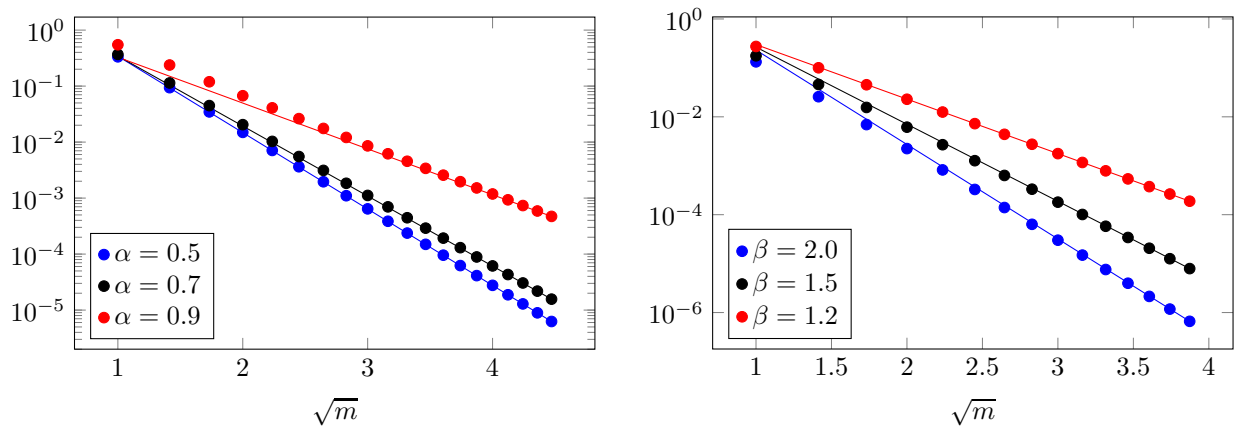
\begin{figure}[ht]
    \centering
    \begin{tikzpicture}[]
\begin{axis}[
  height = {6cm},
  legend pos = {south west},
  ymode = {log},
  xlabel = {$\sqrt{m}$},
  width = {8.5cm}
]

\addplot[
  only marks = {true},
  blue
] coordinates {
  (1.0, 0.33324108417451226)
  (1.4142135623730951, 0.09374364809643937)
  (1.7320508075688772, 0.034611289195575344)
  (2.0, 0.014912288820077068)
  (2.23606797749979, 0.007096952377576238)
  (2.449489742783178, 0.0036269177614726145)
  (2.6457513110645907, 0.001956411651912)
  (2.8284271247461903, 0.0011021938328795722)
  (3.0, 0.0006422863752016994)
  (3.1622776601683795, 0.0003858345749284675)
  (3.3166247903554, 0.00023721774026819187)
  (3.4641016151377544, 0.00014932703115555542)
  (3.605551275463989, 9.563154851129775e-5)
  (3.7416573867739413, 6.234960718143334e-5)
  (3.872983346207417, 4.1253616821834704e-5)
  (4.0, 2.7666278515248166e-5)
  (4.123105625617661, 1.8765750984655526e-5)
  (4.242640687119285, 1.2889888740486022e-5)
  (4.358898943540674, 8.930413430804085e-6)
  (4.47213595499958, 6.255381777403851e-6)
};
\addlegendentry{{}{$\alpha=0.5$}}

\addplot[
  no marks,solid,forget plot,blue
] coordinates {
  (1.0, 0.3416284377465393)
  (1.4142135623730951, 0.09298455671658677)
  (1.7320508075688772, 0.03425793675750612)
  (2.0, 0.014763103385359156)
  (2.23606797749979, 0.00703221426329859)
  (2.449489742783178, 0.0035967022347778836)
  (2.6457513110645907, 0.0019414662346111382)
  (2.8284271247461903, 0.0010936825458203413)
  (3.0, 0.00063797154301453)
  (3.1622776601683795, 0.0003831717698532822)
  (3.3166247903554, 0.00023594236606204417)
  (3.4641016151377544, 0.00014845404577468204)
  (3.605551275463989, 9.51922223594502e-5)
  (3.7416573867739413, 6.207282429452386e-5)
  (3.872983346207417, 4.108880137606947e-5)
  (4.0, 2.7569250114442568e-5)
  (4.123105625617661, 1.8726755867062273e-5)
  (4.242640687119285, 1.2863872811458969e-5)
  (4.358898943540674, 8.927949425587886e-6)
  (4.47213595499958, 6.255381777403851e-6)
};

\addplot[
  only marks = {true},
  black
] coordinates {
  (1.0, 0.3684175832943165)
  (1.4142135623730951, 0.11374345849146472)
  (1.7320508075688772, 0.044847930279903886)
  (2.0, 0.020477816997461356)
  (2.23606797749979, 0.010286907323225474)
  (2.449489742783178, 0.0055214750060664205)
  (2.6457513110645907, 0.003121507110576665)
  (2.8284271247461903, 0.0018380337263825403)
  (3.0, 0.0011153664558819242)
  (3.1622776601683795, 0.0006973002887393786)
  (3.3166247903554, 0.0004454765653104357)
  (3.4641016151377544, 0.00029062088596987267)
  (3.605551275463989, 0.00019280851543389073)
  (3.7416573867739413, 0.0001302078155320354)
  (3.872983346207417, 8.886766145113093e-5)
  (4.0, 6.149343728207363e-5)
  (4.123105625617661, 4.3125889263563636e-5)
  (4.242640687119285, 3.0489882910167458e-5)
  (4.358898943540674, 2.180375315781222e-5)
  (4.47213595499958, 1.5682785495163058e-5)
};
\addlegendentry{{}{$\alpha=0.7$}}

\addplot[
  no marks,solid,forget plot,black
] coordinates {
  (1.0, 0.3445327799558294)
  (1.4142135623730951, 0.10453657311784051)
  (1.7320508075688772, 0.041862177592680454)
  (2.0, 0.019353475400998948)
  (2.23606797749979, 0.009807598242646145)
  (2.449489742783178, 0.005304992989779852)
  (2.6457513110645907, 0.0030148471608285026)
  (2.8284271247461903, 0.0017816992051283395)
  (3.0, 0.0010871447707968206)
  (3.1622776601683795, 0.0006813398311262457)
  (3.3166247903554, 0.0004368748844589355)
  (3.4641016151377544, 0.00028572056457417246)
  (3.605551275463989, 0.00019013528449414086)
  (3.7416573867739413, 0.0001284889837867788)
  (3.872983346207417, 8.803320509677653e-5)
  (4.0, 6.106829539307786e-5)
  (4.123105625617661, 4.2842589339434724e-5)
  (4.242640687119285, 3.036690378186124e-5)
  (4.358898943540674, 2.1728155834530717e-5)
  (4.47213595499958, 1.5682785495163058e-5)
};

\addplot[
  only marks = {true},
  red
] coordinates {
  (1.0, 0.54741930248043)
  (1.4142135623730951, 0.23793764822958596)
  (1.7320508075688772, 0.11947203168327833)
  (2.0, 0.0672695482478351)
  (2.23606797749979, 0.04089708833648367)
  (2.449489742783178, 0.026242996617126226)
  (2.6457513110645907, 0.017534759209986785)
  (2.8284271247461903, 0.012112816857525492)
  (3.0, 0.008564473638493154)
  (3.1622776601683795, 0.0061825267424990464)
  (3.3166247903554, 0.004548241397030996)
  (3.4641016151377544, 0.003396538593778413)
  (3.605551275463989, 0.0025673686432305227)
  (3.7416573867739413, 0.0019630353539327964)
  (3.872983346207417, 0.001518607689228358)
  (4.0, 0.0011844066750980526)
  (4.123105625617661, 0.0009310489557754246)
  (4.242640687119285, 0.0007359799597303819)
  (4.358898943540674, 0.0005881883953689534)
  (4.47213595499958, 0.00047098812514022185)
};
\addlegendentry{{}{$\alpha=0.9$}}

\addplot[
  no marks,solid,forget plot,red
] coordinates {
  (1.0, 0.3276348374655304)
  (1.4142135623730951, 0.15007355147046694)
  (1.7320508075688772, 0.0824355210167374)
  (2.0, 0.04974669830337836)
  (2.23606797749979, 0.03187960428884317)
  (2.449489742783178, 0.021320645836681296)
  (2.6457513110645907, 0.014727721743260652)
  (2.8284271247461903, 0.010437402555764882)
  (3.0, 0.007553329832782846)
  (3.1622776601683795, 0.005562804126758213)
  (3.3166247903554, 0.004158543256799002)
  (3.4641016151377544, 0.003149291694630773)
  (3.605551275463989, 0.002412229058256068)
  (3.7416573867739413, 0.0018663735404220129)
  (3.872983346207417, 0.0014571080201316216)
  (4.0, 0.0011468658927849431)
  (4.123105625617661, 0.000909358633690699)
  (4.242640687119285, 0.0007259065374522404)
  (4.358898943540674, 0.0005830538911160964)
  (4.47213595499958, 0.0004709881251402218)
};

\end{axis}
\end{tikzpicture}
    \quad
    \begin{tikzpicture}[]
\begin{axis}[
  height = {6cm},
  legend pos = {south west},
  ymode = {log},
  xlabel = {$\sqrt{m}$},
  width = {8.5cm}
]

\addplot[
  only marks = {true},
  blue
] coordinates {
  (1.0, 0.13293170309676716)
  (1.4142135623730951, 0.025751418894998254)
  (1.7320508075688772, 0.006941333539715047)
  (2.0, 0.00224243298552107)
  (2.23606797749979, 0.0008224185326983768)
  (2.449489742783178, 0.00032800653298337146)
  (2.6457513110645907, 0.00014001861325163567)
  (2.8284271247461903, 6.358836672792911e-5)
  (3.0, 3.0131321640869402e-5)
  (3.1622776601683795, 1.4932368134212837e-5)
  (3.3166247903554, 7.577058982466411e-6)
  (3.4641016151377544, 3.9748710465525095e-6)
  (3.605551275463989, 2.1373442184556414e-6)
  (3.7416573867739413, 1.1772271454435768e-6)
  (3.872983346207417, 6.617326794628324e-7)
};
\addlegendentry{{}{$\beta=2.0$}}

\addplot[
  no marks,solid,forget plot,blue
] coordinates {
  (1.0, 0.23129172659398395)
  (1.4142135623730951, 0.03672205140787759)
  (1.7320508075688772, 0.008946512187628933)
  (2.0, 0.002720448785269815)
  (2.23606797749979, 0.0009531094556009137)
  (2.449489742783178, 0.00036926769110427537)
  (2.6457513110645907, 0.00015440123836267533)
  (2.8284271247461903, 6.857632799504814e-5)
  (3.0, 3.199786564898481e-5)
  (3.1622776601683795, 1.555976715800372e-5)
  (3.3166247903554, 7.83767569739382e-6)
  (3.4641016151377544, 4.070313626683169e-6)
  (3.605551275463989, 2.171190798622837e-6)
  (3.7416573867739413, 1.1859833390389512e-6)
  (3.872983346207417, 6.617326794628324e-7)
};

\addplot[
  only marks = {true},
  black
] coordinates {
  (1.0, 0.1767859704707097)
  (1.4142135623730951, 0.04583976683494145)
  (1.7320508075688772, 0.015554861922425934)
  (2.0, 0.006148183615519458)
  (2.23606797749979, 0.0026924742681313507)
  (2.449489742783178, 0.0012692782113246496)
  (2.6457513110645907, 0.0006339335253062739)
  (2.8284271247461903, 0.0003321159483150593)
  (3.0, 0.0001803830512185911)
  (3.1622776601683795, 0.00010099633904250087)
  (3.3166247903554, 5.80240538227593e-5)
  (3.4641016151377544, 3.420988088076449e-5)
  (3.605551275463989, 2.0625502778452406e-5)
  (3.7416573867739413, 1.2596875752763436e-5)
  (3.872983346207417, 7.911732763099533e-6)
};
\addlegendentry{{}{$\beta=1.5$}}

\addplot[
  no marks,solid,forget plot,black
] coordinates {
  (1.0, 0.2657676078518324)
  (1.4142135623730951, 0.059146752140622504)
  (1.7320508075688772, 0.01867219825988742)
  (2.0, 0.007064085336888851)
  (2.23606797749979, 0.0030001540842081263)
  (2.449489742783178, 0.001383276726842417)
  (2.6457513110645907, 0.0006787495361230476)
  (2.8284271247461903, 0.0003498755914819249)
  (3.0, 0.00018776291832625603)
  (3.1622776601683795, 0.00010421977187405104)
  (3.3166247903554, 5.953666893696679e-5)
  (3.4641016151377544, 3.4869260925437324e-5)
  (3.605551275463989, 2.087355522118043e-5)
  (3.7416573867739413, 1.2739975168637077e-5)
  (3.872983346207417, 7.911732763099533e-6)
};

\addplot[
  only marks = {true},
  red
] coordinates {
  (1.0, 0.2741538783629729)
  (1.4142135623730951, 0.10049645508867298)
  (1.7320508075688772, 0.04522944691780171)
  (2.0, 0.022867653375797395)
  (2.23606797749979, 0.012522920498012808)
  (2.449489742783178, 0.0072554815885625965)
  (2.6457513110645907, 0.0043909003649096785)
  (2.8284271247461903, 0.002750638733044752)
  (3.0, 0.001770503185398286)
  (3.1622776601683795, 0.0011675306660382302)
  (3.3166247903554, 0.0007877316873275485)
  (3.4641016151377544, 0.0005389979092109253)
  (3.605551275463989, 0.00037459221092506147)
  (3.7416573867739413, 0.0002642700922331098)
  (3.872983346207417, 0.00018873588804053652)
};
\addlegendentry{{}{$\beta=1.2$}}

\addplot[
  no marks,solid,forget plot,red
] coordinates {
  (1.0, 0.29949027506962383)
  (1.4142135623730951, 0.10350116599886078)
  (1.7320508075688772, 0.045800667655345226)
  (2.0, 0.023034251092830617)
  (2.23606797749979, 0.012571651487858182)
  (2.449489742783178, 0.007271745068111106)
  (2.6457513110645907, 0.0043954319897368055)
  (2.8284271247461903, 0.0027510541069815718)
  (3.0, 0.0017715991722409771)
  (3.1622776601683795, 0.0011683871048156593)
  (3.3166247903554, 0.0007863985795955099)
  (3.4641016151377544, 0.000538706561286794)
  (3.605551275463989, 0.000374779755142077)
  (3.7416573867739413, 0.0002643338252458785)
  (3.872983346207417, 0.00018873588804053652)
};

\end{axis}
\end{tikzpicture}
    \caption{Left: The points denote $\epsilon_{\alpha,\alpha}^{[m]}$, while the lines represent $Ce^{-2\pi\sqrt{\alpha(1-\alpha)m}}$ (c.f. \Cref{thm:approx-alpha-opt}) for each $\alpha\in\{0.5,0.7,0.9\}$. Right: The points denote $\epsilon_{1,\beta}^{[m]}$, while the lines represent $Ce^{-2\pi\sqrt{(\beta-1)m/\beta}}$ (c.f. \Cref{conj:alpha-beta-opt}) for each $\beta\in\{1.2,1.5,2\}$.}\label{fig:alpha}
\end{figure}

\subsection{Obtaining best global approximants}\label{sec:computing-best}
    
The best relative approximations (i.e., the minimizing $\tilde{f}$ in \eqref{eq:def-error-alpha}) in our numerical illustrations were computed using the \emph{differential correction} algorithm for uniform rational approximation \cite{DC-orig,DC-conv} -- specifically the version which allows for best \emph{weighted} uniform approximation and linear constraints in the numerator and denominator polynomials \cite{DC-algo-full}. Indeed, the approximation $\tilde{f}$ can be obtained as the best order $[m+1/m]$ rational approximation to $f_\alpha(x)$ relative to the nonnegative function $f_\beta(x)$ (see \Cref{cor:alpha-approx-is-quad}).  Since $f_\alpha(x)$ and $f_\beta(x)$ have double roots at $x=1$, in practice it is preferable to compute the best order $[m-1/m]$ rational approximation to $f_\alpha(x)/(x-1)^2$ relative to the positive function $f_\beta(x)/(x-1)^2$.

\paragraph{The differential correction algorithm}
The differential correction algorithm is an iterative algorithm for finding the best order $[m_1,m_2]$ rational approximation to a function $f(x)$ on an interval $I$, given a discretization $(x_i)_{i=1}^N$  of $I$. At iteration $t+1$, given polynomials $p_t,q_t$ of degree $m_1,m_2$ respectively such that $q_t(x_i)>0$ for every $i\in[N]$, let $\Delta_t=\max_{i\in[N]}\{|f(x_i) - \frac{p_t(x_i)}{q_t(x_i)}|\}$. We aim to find polynomials $p,q$ ``close'' to $p_t,q_t$ such that $\Delta=\max_{i\in[N]}\{|f(x_i) - \frac{p(x_i)}{q(x_i)}|\} < \Delta_t$.
We can rephrase this as 
\begin{equation}\label{eq:dc-initial}
    \min_{p\in\R_{m_1}[x],q\in\R_{m_2}[x],\Delta\in\R}\Delta \quad\text{subject to}\quad \frac{|f(x_i)q(x_i)-p(x_i)|}{q_t(x_i)}\le\Delta\frac{q(x_i)}{q_t(x_i)}.
\end{equation}
The problem \eqref{eq:dc-initial} is almost a linear program in the variables $(p,q,\Delta)$, except that the right-hand side is quadratic. Assuming that the second order term $\frac{(\Delta-\Delta_t)(q-q_t)}{q_t}$ is small, we can linearise the right-hand side 
\[\Delta\frac{q(x_i)}{q_t(x_i)} \approx \Delta_t\frac{q(x_i)}{q_t(x_i)} + (\Delta-\Delta_t),\]
to get the iteration
\begin{equation}\label{eq:dc-lp}
(p_{t+1},q_{t+1})\in\argmin_{p\in\R_{m_1}[x],q\in\R_{m_2}[x]}\max_{i\in[N]} \frac{|f(x_i)q(x_i)-p(x_i)| - \Delta_tq(x_i)}{q_t(x_i)}\quad\text{subject to}\quad\Vert q\Vert_\infty\le 1.
\end{equation}
Here $\Vert q\Vert_\infty$ is the  $\ell_\infty$ norm of the coefficients of $q$, so \eqref{eq:dc-lp} is a linear program.
The normalization condition $\Vert q\Vert_\infty\le 1$ is necessary since the objective function is homogeneous in $(p,q)$. Although this derivation was rather informal, it can be proved that $(p_t,q_t,\Delta_t)$  form a minimizing sequance \cite{DC-conv}.
To find the best rational approximation relative to a function $b(x)>0$,
\eqref{eq:dc-lp} is modified to
\begin{equation}
    \label{eq:dc-lp-w}
    (p_{t+1},q_{t+1})\in\argmin_{p\in\R_{m_1}[x],q\in\R_{m_2}[x]}\max_{i\in[N]} \frac{|f(x_i)q(x_i)-p(x_i)| - \Delta_tq(x_i)b(x_i)}{q_t(x_i)b(x_i)}\quad\text{subject to}\quad\Vert q\Vert_\infty\le 1.
\end{equation}

\paragraph{Comments on the practical implementation}
Our implementation is made available at \codeurl.
The main reason for choosing the differential correction algorithm (instead of, e.g., the Remez algorithm) is that it is guaranteed to converge for any feasible initialization. A drawback of the differential correction algorithm which is often mentioned is that, since the approximation domain must be discretized, and the linear program solved at each iteration scales with the size of the discretization, it can be quite slow. However, since we know in advance that the only singularities of the functions $f_\alpha(x)/(x-1)^2$ are at $x=0$ and at $x=\infty$, we are free to choose discretizations with points exponentially distributed near these points. This means that the number of discretization points in total can be quite modest ($\sim500$) and still give very accurate results. 

Good numerical stability was obtained by representing the rational approximants using barycentric coordinates, as described in \cite{rational-bary}. 

\begin{figure}
    \centering
    \input{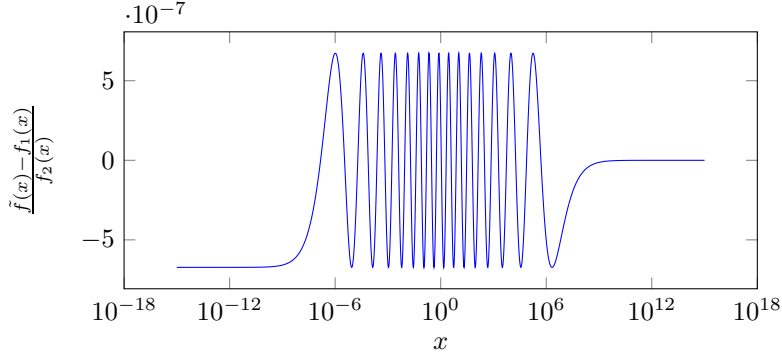}
    \caption{Error (relative to $f_2$) of the 15-node quadrature rule approximation to $f_1$ attaining the optimal accuracy $\epsilon^{[15]}_{1,2}$. Note the equioscillation.}\label{fig:osc}
\end{figure}

\begin{figure}
    \centering
    \begin{tikzpicture}[]
\begin{axis}[
  height = {5cm},
  xmin = {-6},
  xmax = {6},
  ymax = {2.5},
  xlabel = {$\log\left( \frac{t}{1-t}\right)$},
  ymin = {0.5},
  width = {11cm},
  ytick=\empty
]

\addplot+[
  only marks = {true}
] coordinates {
  (1.83935083622747037346646030614926233e+00, 2.00000000000000000000000000000000000e+00)
  (1.30409897498729676920926854550647511e+00, 2.00000000000000000000000000000000000e+00)
  (9.65390878476263367275439436528462911e-01, 2.00000000000000000000000000000000000e+00)
  (7.08728996742336079242478397666746248e-01, 2.00000000000000000000000000000000000e+00)
  (4.95060465371650108496251593500512020e-01, 2.00000000000000000000000000000000000e+00)
  (3.05627925401049257350469179412398600e-01, 2.00000000000000000000000000000000000e+00)
  (1.29305440663913434705868492218666587e-01, 2.00000000000000000000000000000000000e+00)
  (-4.18247867267794723910187124646930446e-02, 2.00000000000000000000000000000000000e+00)
  (-2.14565528103294855133045326049166949e-01, 2.00000000000000000000000000000000000e+00)
  (-3.96043991201692360146799717072989932e-01, 2.00000000000000000000000000000000000e+00)
  (-5.95325147804098795520197486003857880e-01, 2.00000000000000000000000000000000000e+00)
  (-8.26225106961446622661041123633641963e-01, 2.00000000000000000000000000000000000e+00)
  (-1.11404562860401801297640750907648197e+00, 2.00000000000000000000000000000000000e+00)
  (-1.51744211976134403739534846115545791e+00, 2.00000000000000000000000000000000000e+00)
  (-2.24621119136323001437733319442420695e+00, 2.00000000000000000000000000000000000e+00)
};

\addplot+[
  only marks = {true}
] coordinates {
  (5.060740872621069, 1.0)
  (3.8641233216551947, 1.0)
  (2.9688826854056827, 1.0)
  (2.222847482019495, 1.0)
  (1.5724356819062655, 1.0)
  (0.9924121974360661, 1.0)
  (0.4655350151305856, 1.0)
  (-0.030992980777400564, 1.0)
  (-0.5297060803121784, 1.0)
  (-1.062258749996775, 1.0)
  (-1.6501997340428893, 1.0)
  (-2.310949049725883, 1.0)
  (-3.071900462214871, 1.0)
  (-3.9929720606930394, 1.0)
  (-5.256552720377935, 1.0)
};

\end{axis}
\end{tikzpicture}
    \caption{The nodes $t_i$ corresponding to the order $[16/15]$ Pad\'e approximation of $f_1$ (blue discs). Below (red squares) are the nodes of the 15-node quadrature rule approximation to $f_1$ attaining the optimal accuracy $\epsilon^{[15]}_{1,2}$ relative to $f_2$.}\label{fig:nodes}
\end{figure}
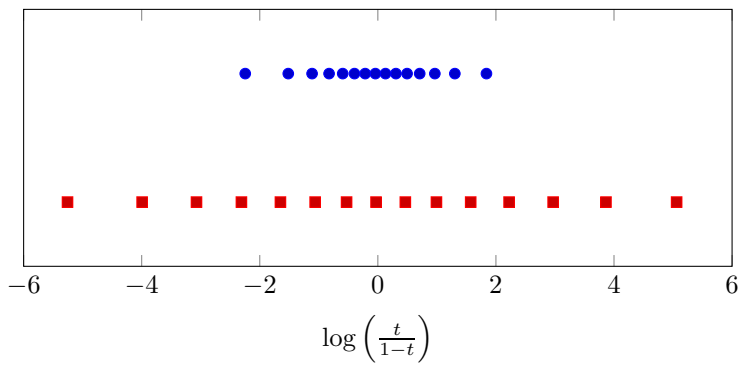

\section*{Acknowledgments}

HF would like to thank Y. Nakatsukasa and L.N. Trefethen for discussions about best rational approximations. We would also like to thank J. Saunderson for comments that helped improve the exposition.
HF acknowledges funding from UK Research and Innovation (UKRI) under the UK government’s Horizon Europe funding guarantee EP/X032051/1.

\bibliographystyle{alpha}
\bibliography{paper}

\appendix

\section{Proof of Theorem \ref{thm:root-exp}}
\label{sec:proof-root-exp}

    We split $G$ into
    \[G(w) = \underbrace{\int_{-1}^{-1/2} \frac{\phi(\lambda)}{1-\lambda w}\d \lambda}_{G^-(w)} + \underbrace{\int_{-1/2}^{1/2} \frac{\phi(\lambda)}{1-\lambda w}\d \lambda}_{G^0(w)} + \underbrace{\int_{1/2}^{1} \frac{\phi(\lambda)}{1-\lambda w}\d \lambda}_{G^+(w)}.\]
    Note that $G^0$ is analytic on $(-2,2)$, so $E_{m,m}(G^0,[-1,1])=O(\rho^m)$ for some $\rho\in(0,1)$.
    Applying Theorems 1 and 2 from \cite{Pekarskii1995} to each of $G^-$ and $G^+$,
    we obtain
    \[E_{m,m}(G^-,[-1,1])=O(e^{-2\pi\sqrt{\beta m}}),\quad E_{m,m}(G^+,[-1,1])=O(e^{-2\pi\sqrt{\alpha m}}).\]
    Note that the sum of two rational functions, of order $[m_1/m_1]$ and $[m_2/m_2]$ respectively, is a (possibly degenerate) rational function of order $[m_1+m_2/m_1+m_2]$. Therefore
    \begin{align*}
        E_{m,m}(G^-+G^+,[-1,1])  &\le E_{\lfloor \frac{\alpha m}{\alpha+\beta} \rfloor,\lfloor \frac{\alpha m}{\alpha+\beta} \rfloor}(G^-,[-1,1]) + E_{\lfloor \frac{\beta m}{\alpha+\beta} \rfloor,\lfloor \frac{\beta m}{\alpha+\beta} \rfloor}(G^+,[-1,1])\\
        &=O\left(e^{-2\pi\sqrt{\frac{\alpha\beta m}{\alpha+\beta }}}\right),
    \end{align*}
    so $E_{m,m}(G^-+G^+,[-1,1])=O(e^{-2\pi\sqrt{\kappa m}})$.

    For any $\theta>0$ and all $m>\theta^2$,
    \begin{align*}
        E_{m,m}(f,[-1,1]) &\le E_{\lceil m-\theta\sqrt{m} \,\rceil,\lceil m-\theta\sqrt{m}\, \rceil}(G^-+G^+,[-1,1]) + E_{\lfloor \theta\sqrt{m} \rfloor ,\lfloor \theta\sqrt{m} \rfloor}(G^0,[-1,1])\\
        &=O(e^{-2\pi\sqrt{\kappa(m-\theta\sqrt m)}})  + O(\rho^{\theta\sqrt{m}}).
    \end{align*}
    Note that $\sqrt{m} - \sqrt{m-\theta\sqrt{m}}\to0$ as $m\to\infty$ and choosing $\theta=\frac{-2\pi\sqrt{\kappa}}{\log\rho}$, we obtain
    \[E_{m,m}(f,[-1,1]) = O(e^{-2\pi\sqrt{\kappa m}})\]
    as required.

\section{Proof of \Cref{thm:op-convex-quad} that best rational approximants of operator convex functions have quadrature form }\label{sec:proof-osc}

We first need to recall some known  facts on the connection between best rational approximants, equioscillation and rational interpolants.

\subsection{Equioscillation and rational interpolants}

\begin{definition}\label{def:equiosc}
    Let $I$ be an interval in $\R$. Let $f:I\to\R$ be continuous such that $\epsilon := \sup_{x\in I}|f(x)|<\infty$.
    We say $f$ \emph{equioscillates}  between $n$ points on $I$ if there exist $x_1<\dots<x_n$ in $I$ such that $f(x_j)=(-1)^{j+i}\epsilon$, where $i=0$ or $1$.
\end{definition}
\begin{definition}\label{def:defect}
    Let $r\in\cR_{m_1,m_2}$. Write $r=p/q$, with $p,q$ having no roots in common. The \emph{defect} of $r$ in $\cR_{m_1,m_2}$ is
    \[d=\min\{m_1-\deg p, m_2 - \deg q\}.\]
    If $r=0$, we say $d=m_2$.
\end{definition}
In approximation theory, best uniform approximations are often characterized by equioscillation of the error, in both the unweighted and weighted case. We will need the following extension of a result proved in \cite{akhiezer1956theory}. It is an extension because we allow the weight function $S$ to vanish at the boundary of the approximation domain. 
\begin{theorem}\label{thm:weighted-osc}
    Let $F,S:[-1,1]\to\R$ be continuous, with $S$ positive on $(-1,1)$, and let $m_1,m_2$ be nonnegative integers. There exists $\tilde{R}\in\cR_{m_1,m_2}$ such that
    \begin{equation}\label{eq:S-weighted-problem}
        \sup_{w\in[-1,1]}|F(w) - S(w)\tilde{R}(w)| = \inf_{R\in\cR_{m_1,m_2}}\sup_{w\in[-1,1]}|F(w) - S(w)R(w)|.
    \end{equation}
    Assume further that, if $S$ vanishes at either endpoint $a=\pm1$ then it does so slowly, in the sense that
    \begin{equation}\label{eq:endpoint-inf}
            \lim_{w\to a}\frac{S(w)}{|w-a|}=\infty.
    \end{equation}
    Then $F-S\tilde{R}$ equioscillates between $m_1+m_2+2-d$ points on $[-1,1]$, where $d$ is the defect of $\tilde{R}$ in $\cR_{m_1,m_2}$.
\end{theorem}
\begin{remark}
    The problem of minimizing the right-hand side of \eqref{eq:S-weighted-problem} over $R$ is essentially the same as finding the best rational approximation to $F/S$ relative to the function $1/S$ on $w\in(-1,1)$.
    We note that a similar extension is proved in \cite{DUNHAM-vanishing}, but with the requirement that $F/S$ be continuous (i.e. has finite limits at $w=\pm1$).
\end{remark}
\begin{proof}
    In the case where $S$ is positive on $(-1,1)$, this is proved in \cite[§33--§34]{akhiezer1956theory}. The proof of the first statement (§33 in \cite{akhiezer1956theory}) does not actually require $S$ to be positive at $\{\pm1\}$.

    For the proof of the second statement, write $\tilde{R} = P/Q$, where $P,Q$ have no roots in common. The proof of equioscillation in \cite[§34]{akhiezer1956theory} uses positivity of $S$ at $\{\pm1\}$ only once, in order to conclude that $Q$ has no root in $[-1,1]$. Indeed, if this were so, then $S\cdot\tilde{R}$ would become infinite on $[-1,1]$.
    
    Since we allow $S$ to vanish at $\pm1$, in principle $Q$ could have a zero where $S$ vanishes, and $S\tilde{R}$ could remain bounded. But in fact, this is impossible because of our assumption \eqref{eq:endpoint-inf}. So the argument from \cite{akhiezer1956theory} remains valid, with this adjustment.
\end{proof}

We will also need the following result, which is a special case of \cite[Theorem V.3.5]{Braess12}, giving a sufficient condition for rational \emph{interpolants} to be of quadrature type.

\begin{theorem}{{\cite[Theorem V.3.5]{Braess12}}}
    \label{thm:interpolants-quadrature}
    Let $G$ be a function of the form $G(w)=\int_{-1}^1 \frac1{1-\lambda w}\d\mu(\lambda)$ for any finite Borel measure $\mu$.
    Suppose that $\tilde{G}\in\cR_{m-1,m}$ interpolates $G$ at $2m$ points in $(-1,1)$, and that $m\le | \supp \mu|$. Then $\tilde{G}$ has the form
    \[\tilde{G}(w) = \sum_{i=1}^m \frac{a_i}{1-\lambda_i w},\]
    for $a_i\ge0$ and $\lambda_i\in[-1,1]$. 
\end{theorem}

\subsection{Proof of \Cref{thm:op-convex-quad}}

We are now ready to prove \Cref{thm:op-convex-quad}. We map the problem to one of the form \eqref{eq:S-weighted-problem} on $(-1,1)$. We claim that $\tilde{f} \in \cR_{m+1,m}$ is a solution of
\begin{equation}\label{eq:b2-x0}
\inf_{r \in \cR_{m+1,m}} \sup_{x \in (0,\infty)} \left|\frac{f(x)-r(x)}{b(x)}\right|
\end{equation}
if, and only if, the function $\tilde{F}(w) := (1+x) \tilde{f}(x)/(x-1)^2$ with $x=\frac{1+w}{1-w}$ is a solution of
\begin{equation}\label{eq:b2-w1}
\inf_{R \in \cR_{m-1,m}} \sup_{w \in (-1,1)} |S(w) F(w) - S(w) R(w)|
\end{equation}
where $F(w) := (1+x) f(x)/(x-1)^2$ and $S(w) := (x-1)^2/((1+x)b(x))$. Indeed, by construction we have 
\[\sup_{x \in (0,\infty)} \left|\frac{f(x)-\tilde{f}(x)}{b(x)}\right| = \sup_{w \in (-1,1)} |S(w) F(w) - S(w) \tilde{F}(w)|,\] 
so we just need to check that if $\tilde{f} \in \cR_{m+1,m}$, then $\tilde{F}\in\cR_{m-1,m}$ and vice versa. First suppose $\tilde{f} \in \cR_{m+1,m}$ is a minimizer in \eqref{eq:b2-x0}. Then $\tilde{f}$ must have a double root at 1, so $\tilde f(x)/(x-1)^2\in\cR_{m-1,m}$. Let $p,q$ be polynomials of degree $m-1,m$ respectively with $\tilde f(x)/(x-1)^2=p(x)/q(x)$.
\begin{align*}
    \tilde{F}(w) = (1+x)\tilde f(x)/(x-1)^2 = (1+x)\frac{p(x)}{q(x)} 
    &= \left(\frac{2}{1-w}\right)\frac{p(\frac{1+w}{1-w})}{q(\frac{1+w}{1-w})}\\
    &= \left(\frac{2}{1-w}\right)\frac{(1-w)^{m-1}p(\frac{1+w}{1-w})}{(1-w)^{m-1}q(\frac{1+w}{1-w})}\\
    &= \frac{2(1-w)^{m-1}p(\frac{1+w}{1-w})}{(1-w)^{m}q(\frac{1+w}{1-w})} \in \cR_{m-1,m}.
\end{align*} Similarly, it can be verified that if $\tilde{F}(w)\in\cR_{m-1,m}$, then $\tilde{f}(x) = (x-1)^2 \tilde{F}\bigl(\frac{x-1}{x+1}\bigr)/(1+x)$ is a rational function of order $[m+1/m]$.
\\
Note that the function $S\cdot F$ is continuous on $(-1,1)$, using \eqref{eq:h-cont-mid}, and by assumption, approaches finite limits at $\pm1$, by \eqref{eq:fh-cont-ends}. Therefore, it can be  extended to a continuous function on $[-1,1]$. Similarly, by \eqref{eq:h-cont-ends}, $S$ can be extended to a continuous function on $[-1,1]$. Therefore, \eqref{eq:b2-w1} is equivalent to
\begin{equation}
    \inf_{R \in \cR_{m-1,m}} \sup_{w \in [-1,1]} |S(w) F(w) - S(w) R(w)|
\end{equation}
(i.e. with closed interval $[-1,1]$).
By \eqref{eq:h-slowly-vanish},
\[\lim_{w\to-1^+}\frac{S(w)}{w+1}=\lim_{x\to0^+}\frac{1}{2xb(x)}=\infty,\quad \lim_{w\to1^-}\frac{S(w)}{1-w}=\lim_{x\to\infty}\frac{x^2}{2b(x)}=\infty.\]
    \\
We can therefore apply \Cref{thm:weighted-osc}, to conclude that a minimizer $\tilde{F}$ of \eqref{eq:b2-w1} exists, and that $SF-S\tilde{F}$ equioscillates between between $2m+1-d$ points in $[-1,1]$, where $d$ is the defect of $\tilde{F}$ in $\cR_{m-1,m}$. 
\\
Since $f$ is operator convex, we have
\[F(w) = \frac{(1+x) f(x)}{(x-1)^2}=\int_0^1\frac{2}{1-(1-2t)w}\d\mu(t)\]
for a finite measure $\mu$.
Observe that, by \Cref{def:defect}, $\tilde{F}\in\cR_{m-d-1,m-d}$. The equioscillation of $S(F-\tilde{F})$ implies that $\tilde{F}$ interpolates $F$ at $2m-d\ge 2m-2d$ points in $(-1,1)$. This allows us to conclude, from \Cref{thm:interpolants-quadrature}, that $\tilde F$ is a quadrature rule for 
$F(w)$
i.e.
\[\tilde F(w)\equiv\sum_{i=1}^m\frac{u_i}{1-(1-2t_i)w}\]
for weights $u_i\ge0$ and nodes $0\le t_1<\dots < t_m \le 1$.
By inverting the transformation that defined $\tilde{F}$ in terms of $\tilde{f}$, we obtain \eqref{eq:f-def-op-quad}.
\qedhere

\section{Lemmas needed for \Cref{thm:approx-alpha-opt} and \Cref{thm-alpha-beta-subopt}}

\begin{lemma}\label{lem:im-bound}
    For every $x>0$, and $z$ in the strip $|\Im(z)|<\pi$, we have
    \[|\omega_x(z)| \le \omega_x(\Re(z)) \cos^{-3}\left(\frac{\Im(z)}{2}\right),\]
    where $\omega_x$ is as defined in \eqref{eq:def-omega-x}.
\end{lemma}
\begin{proof}
    Let us write $e^z=re^{\theta i}$, where $r=e^{\Re(z)}$ and $\theta=\Im(z)$.
    Then 
    \begin{align*}
        |1+e^z|^2 &= 1+2r\cos\theta + r^2\\
        &= (1+r)^2 - 2r(1-\cos\theta)\\
        &\ge (1+r)^2[1 - \frac{1-\cos\theta}{2}]\\
        & = (1+r)^2\cos^2(\theta/2).
    \end{align*}
    By the same calculation, but with the substitution $xr\mapsto r$, we have
    \[|1+xe^z|^2 \ge (1+xr)^2\cos^2(\theta/2).\]
    Therefore, 
    \[|\omega_x(z)| \le \frac{r}{|1+xe^z||1+e^z|^2} \ge  \frac{r}{(1+xr)(1+r)^2}\cos^{-3}\left(\frac{\theta}{2}\right) =\omega_x(\Re(z))\cos^{-3}\left(\frac{\theta}{2}\right). \]
\end{proof}

\begin{lemma}
    \label{lemma:f-incr-in-alpha}
    For every $x>0$, $\alpha f_\alpha(x)$ and  $(\alpha-1)f_\alpha(x)$ are increasing in $\alpha$, for $\alpha\in[-1,2]$.
\end{lemma}
\begin{proof}
    We have \begin{align*}
        \frac{\d}{\d\alpha}[\alpha f_\alpha(x)]&=\frac{(\alpha-1)[x^\alpha\log x - x] - [x^\alpha-\alpha x]}{(\alpha-1)^2}\\
        &=\frac{x}{(\alpha-1)^2}[x^{\alpha-1}\log(x^{\alpha-1})+1-x^{\alpha-1}]\\
        &=\frac{x}{(\alpha-1)^2}f_1(x^{\alpha-1})\\
        &\ge0.
    \end{align*}
    Next, 
    \[\frac{\d}{\d\alpha}[(\alpha-1) f_\alpha(x)] = -\frac{\d}{\d\alpha}[(1-\alpha) xf_{1-\alpha}(\frac1x)] = -x\frac{\d (1-\alpha)}{\d\alpha}\frac{\d}{\d\gamma}[\gamma f_\gamma(\frac1x)]|_{\gamma=1-\alpha}\ge0.\]
\end{proof}

\begin{lemma}\label{lem:helper-bound-rational}
    Let $\alpha\in[-1,2]$. Then we have, for all $x>0$,
    \begin{equation}\label{eq:lem-alpha-bound}
    \frac{1}{3} \le \frac{1+x}{(x-1)^2}\,f_\alpha(x).
    \end{equation}
\end{lemma}
\begin{proof}
    Recall there is a nonnegative measure $\mu$ such that $f_\alpha(x)=\int_0^1\frac{(x-1)^2}{1+t(x-1)}\d\mu(t)$. We have $\frac{1+x}{(x-1)^2}\,f_\alpha(x)=\int_0^1\frac{1+x}{1+t(x-1)}\d\mu(t)$. Consider a change of variables $x=\frac{1+w}{1-w}$, $w\in(-1,1)$, and let $G_\alpha(w):=\frac{1-w}{2w^2}\,f_\alpha(\frac{1+w}{1-w})$ be the resulting function on $(-1,1)$. We have $G_\alpha(w)=\int_0^1\frac{2}{1+(2t-1)w}\d\mu(t)$. Since each of the functions $w\mapsto \frac{2}{1+(2t-1)w}$ is convex, it follows that $G_\alpha(w)$ is convex, i.e. $G_\alpha$ lies above each of its tangents.

    We list in the table below the values of $G_\alpha$ and its derivative at $w=0$, and as $w\to -1^+$, $w\to1^-$.
    \begin{center}
        \begin{tabular}{|c|c|c|c|c|c|c|c|}
            \hline
            $w$ & \multicolumn{3}{c|}{$-1^+$} & $0$ & \multicolumn{3}{c|}{$1^-$} \\
            \hline
            $\alpha$   & $[-1,0]$ & $(0,1]$ & $(1,2]$ & $[-1,2]$ & $[-1,0)$ & $[0,1)$ & $[1,2]$ \\  
            \hline
            $G_\alpha(w)$  & $\infty$ & $\frac1\alpha$ & $\frac1\alpha$ & 1 & $\frac1{1-\alpha}$ & $\frac1{1-\alpha}$ & $\infty$\\
            $G_\alpha'(w)$  & $-\infty$ & $-\infty$ & $\frac{2\alpha-3}{2\alpha(\alpha-1)}$ & $\frac{2\alpha-1}{3}$ & $\frac{2\alpha+1}{2\alpha(\alpha-1)}$ & $\infty$ & $\infty$\\
            \hline
        \end{tabular}
    \end{center}
    Observe that for $\alpha\geq\frac32$, $G_\alpha'(-1)\ge0$. Therefore, in this case,  for all $w\in(-1,1)$ $G_\alpha(w)\ge G_\alpha(-1)=\frac1\alpha\ge\frac23>\frac13$.
    Similarly, if $\alpha\leq-\frac12$, $G_\alpha'(1)\le0$. Therefore, in this case,  for all $w\in(-1,1)$ $G_\alpha(w)\ge G_\alpha(1)=\frac1{1-\alpha}\ge\frac23>\frac13$.

    Now consider the remaining case $\alpha\in(-\frac12,\frac32)$. We have $|G_\alpha'(0)|=\left|\frac{2\alpha-1}{3}\right|\le\frac23$. Therefore, for all $w\in(-1,1)$, 
    \[G_\alpha(w)\ge G_\alpha(1) - wG_\alpha'(0) \ge  G_\alpha(1) - |G_\alpha'(0)| \ge \frac13.\]
    This establishes \eqref{eq:lem-alpha-bound}.
\end{proof}

\begin{lemma}\label{lem:helper-bound-rational-comp}
    For $\alpha\in[1,2)$, \Cref{lem:helper-bound-rational} can be strengthened to
    \begin{equation*}
        \frac13\frac{\min\{1, y^{2-\alpha}\}}{1+xy} \le \frac{f_\alpha(x)}{(x-1)^2},
    \end{equation*}
    uniformly in $x,y>0$.
\end{lemma}
\begin{proof}
    For convenience, define the function $g_\alpha(x):=\frac{f_\alpha(x)}{(x-1)^2}$. We must prove
    \begin{equation}\label{eq:lem-beta-bound}
        \frac13\frac{\min\{1, y^{2-\alpha}\}}{1+xy} \le g_\alpha(x).
    \end{equation}

    We may assume that $y$ has been chosen to maximise the left-hand side, that is $y=\min\{1, \frac{2-\alpha}{(\alpha-1)x}\}$. If $x<\frac{2-\alpha}{\alpha-1}$, then $y=1$ and \eqref{eq:lem-beta-bound} is implied by \Cref{lem:helper-bound-rational}. Otherwise, $y=\frac{2-\alpha}{(\alpha-1)x}$ and $x\ge\frac{2-\alpha}{\alpha-1}$, and we necessarily have $\alpha\in(1,2)$. Substituting this into \eqref{eq:lem-beta-bound}, we need to show that
    \begin{equation}\label{eq:lem-bound-inter1}
        \frac13(\alpha-1)^{\alpha-1}(2-\alpha)^{2-\alpha}x^{\alpha-2} \le g_\alpha(x)\quad\text{ for every }x\ge\frac{2-\alpha}{\alpha-1}.
    \end{equation}
    
    We claim that $g_\alpha(x)x^{2-\alpha}$ is nondecreasing in $x$. From this claim, together with \eqref{eq:lem-alpha-bound} for the particular value $x=\frac{2-\alpha}{\alpha-1}$, we immediately deduce \eqref{eq:lem-bound-inter1}. We now prove the claim. Let $\mu$ be the nonnegative measure (of total mass $\frac12$) for which $g_\alpha(x)=\int_0^1\frac1{1+t(x-1)}\d\mu(t)$. We have
    \begin{align*}
        (x-1)g_\alpha'(x) &= \int_0^1\frac{-t(x-1)}{[1+t(x-1)]^2}\d\mu(t)\\
        &=\int_0^1\frac{\d}{\d t}\left(\frac1{1+t(x-1)}\right)t\d\mu(t)\\
        &=\left[\frac{t}{1+t(x-1)}\frac{\d\mu}{\d t}\right]_0^1 - \int_0^1\frac{1}{1+t(x-1)}\frac{\d}{\d t}\left(t\frac{\d\mu}{\d t}\right)\d t\\
        &= - \int_0^1\frac{(2-\alpha) - \frac{\alpha t}{1-t}}{1+t(x-1)}\d\mu(t) && {\scriptstyle{\text{[since }\frac{\d\mu}{\d t}\,\propto \,t^{1-\alpha}(1-t)^\alpha}\text{]}}\\
        &= (\alpha-2)g_\alpha(x) +\alpha\int_0^1\frac{t}{(1-t)(1+t(x-1))}\d\mu(t)\\
        &=(\alpha-2)g_\alpha(x) +\alpha\int_0^1\frac{1/x}{1-t} - \frac{1/x}{1+t(x-1)}\d\mu(t)\\
        &=(\alpha-2)g_\alpha(x) +\frac{\alpha(g_\alpha(0) - g_\alpha(x))}{x}\\
        &=(\alpha-2)g_\alpha(x) +\frac{1 - \alpha g_\alpha(x)}{x}.
    \end{align*}
    Rearranging, we obtain 
    \begin{align*}
        \frac{1}{x-1}\left[1-2g_\alpha(x)\right] &= (2-\alpha)g_\alpha(x) + xg_\alpha'(x)\\
        &= x^{\alpha-1}\frac{\d}{\d x}[g_\alpha(x)x^{2-\alpha}].
    \end{align*}
    Note that
    $\frac{1}{x-1}\left[1-2g_\alpha(x)\right]=\int_0^1\frac{2t}{1+t(x-1)}\d\mu(t)\ge0$.
    This proves that that $g_\alpha(x)x^{2-\alpha}$ is nondecreasing in $x$, concluding the proof.
\end{proof}

\end{document}